\documentclass[]{article}
\usepackage{etex}
\reserveinserts{28}

\usepackage{amscd, amsmath, amssymb, epic, eepic,  delarray,stmaryrd,mathbbol,amsbsy,amsfonts,mathabx}

\usepackage[latin1]{inputenc}
\usepackage{graphicx,epsfig,epsf,epsfig}
\usepackage{pstricks,pst-node,rotating}
\usepackage{color}
\usepackage{tikz}
\usepackage{pgf}
\usepackage[ps,curve,2cell]{xypic}
\usepackage{algorithm}
\usepackage[noend]{algpseudocode}
\usepackage[nottoc]{tocbibind}
\changenotsign
\usetikzlibrary{arrows,automata}
\tikzstyle{my loopup}=[->, to path={
 .. controls +(50:1) and +(100:1) .. (\tikztotarget) \tikztonodes}]

\tikzstyle{my loopdown}=[->, to path={
 .. controls +(-110:1) and +(-30:1) .. (\tikztotarget) \tikztonodes}]

\tikzstyle{my arcdown}=[->, to path={
 .. controls +(-100:1) and +(-120:1) .. (\tikztotarget) \tikztonodes}]

\tikzstyle{my arcup}=[->, to path={
 .. controls +(100:1) and +(120:1) .. (\tikztotarget) \tikztonodes}]

\usetikzlibrary{matrix}

\newcommand\Omit[1]{}

\begin{document}

%%%%%%%%%%%%%%%%
% Environnement %
%%%%%%%%%%%%%%%%% 

%setcounter{secnumdepth}{7}

\newtheorem{definition}{Definition}[section]
\newtheorem{notation}[definition]{Notation}
\newtheorem{lemma}[definition]{Lemma}
\newtheorem{proposition}[definition]{Proposition}
\newtheorem{theorem}[definition]{Theorem}
\newtheorem{corollary}[definition]{Corollary}
\newtheorem{fact}[definition]{Fact}
\newtheorem{example}[definition]{Example}
\newtheorem{constraint}[definition]{Constraint}
\newtheorem{convention}[definition]{Convention}
\newtheorem{remark}[definition]{Remark}

\newenvironment{proof}{\noindent {\bf Proof}}{\hfill$\Box$
\medskip
}

\newtheorem{axiom}{Axiom}

\newcommand{\Introduction}{\section*{Introduction}
                           \addcontentsline{toc}{chapter}{Introduction}}

\def\sl#1{\underline{#1}}
\def\cl#1{{\mathcal{#1}}}

\newcommand{\df}[1]{#1\!\!\searrow}
\newcommand{\udf}[1]{#1\!\!\nearrow}

\newcommand{\rmq}{\hspace{-5mm}{\bf Remark:} ~}
\newcommand{\lpar}{\par\noindent}

%###############################################################
%##### ENVT POUR CONSTRUIRE DES ARBRES DE PREUVE ###############
%###############################################################

\newenvironment{infrule}{\begin{array}{c}}{\end{array}}

\def\rulename#1{({\sc #1})}

\def\et{\hskip 1.5em \relax}

\def\nm{\vspace{-1mm} \\  \hspace{-0.6cm}}
\def\nom{\vspace{-1mm} \\  \hspace{-1.1cm}}
\def\nomq{\vspace{-1mm} \\  \hspace{-1.3cm}}
\def\nomter{\vspace{-1mm} \\  \hspace{-1.4cm}}
\def\sansnom{\vspace{-2mm} \\ \hspace{-0.2cm}}
\def\nombis{\vspace{-1mm} \\  \hspace{-1.8cm}}
\def\decal{\vspace{-2mm} \\ \hspace{0.4cm}}

\def\imp{ ~ \hrulefill \\}

\def\ligne{\\[5mm]}

\def\Ligne{\\[7mm]}

%###########################################

\newenvironment{ex.}{
        \medskip
        \noindent {\bf Example}}
        {\hfill$\Box$ \medskip}
\newenvironment{Enonce}[1]{
        \begin{description}
        \item[{\bf #1~:}]
        \mbox{}
}{
        \end{description}
}

%%%%%%%%%%%%%%%%%%
% Macro generale %
%%%%%%%%%%%%%%%%%%
%\def\Doubleunion#1#2#3{\displaystyle{\bigcup_{\scriptstyle #1 \atop
%\scriptstyle #2}}#3}
\def\Doubleunion#1#2#3{\displaystyle{\bigcup_{#1}^{#2} {#3}}}
\def\Union#1#2{\displaystyle{\bigcup_{#1} #2}}
\def\Intersect#1#2{\displaystyle{\bigcap_{#1} #2}}
\def\Coprod#1#2{\displaystyle{\coprod_{#1} #2}}
\def\Produit#1#2#3{\displaystyle{\prod_{#1}^{#2} {#3}}}
\def\Conj#1#2{\displaystyle{\bigwedge_{#1} \!\!\!\!\!#2}}
\def\Disj#1#2{\displaystyle{\bigvee_{#1} \!\!\!\!\!#2}}

\newcommand{\Nat}{I\!\!N}
\newcommand{\Int}{Z\!\!\! Z}

%%r\`egle d'inf\'erence
\def\infer#1#2{\mbox{\large  ${#1} \frac {#2}$ \normalsize}}

\def\sem#1{[\![ #1 ] \!]}

%%%%%%%%%%%%%%%%%%%%%%%%%%%%%%
% Mot vide sur tout alphabet %
%%%%%%%%%%%%%%%%%%%%%%%%%%%%%%

\newcommand{\mv}{\varepsilon}

%%%%%%%%%%%%%%%%%%%%
% Relation binaire %
%%%%%%%%%%%%%%%%%%%%
\def\M#1#2{M_{\scriptstyle #1 \atop \scriptstyle #2}}
\def\H#1#2{H_{\scriptstyle #1 \atop \scriptstyle #2}}
\newcommand{\Ro}{{\cal{R}}}
\newcommand{\Mo}{{\cal{M}}}
\newcommand{\Ho}{{\cal{H}}}
\newcommand{\Eo}{{\cal{E}}}
\newcommand{\ssucc}{\succ \!\! \succ}

%%%%%%%%%%%%
% Identite %
%%%%%%%%%%%%

\def\id#1{\underline{#1}}

%============================= Title and Abstract

\title{Abstract Mathematical morphology based on structuring element: {Application to morpho-logic}}

\author{Marc Aiguier$^1$ and Isabelle Bloch$^2$ and Ram\'on Pino-P\'erez$^3$\\ \ \\
1. MICS, CentraleSupelec, Universit\'e Paris Saclay, France \\ {\it marc.aiguier@centralesupelec.fr} \\
2. LTCI, T\'el\'ecom Paris, Institut Polytechnique de Paris, France \\ {\it isabelle.bloch@telecom-paris.fr} \\
3. Departemento de Matematicas, Facultad de Ciencias, \\ Universidad de Los Andes, M\'erida, Venezuela  \\ {\it pinoperez@gmail.com}}
\date{}

\maketitle

\begin{abstract}
A general definition of mathematical morphology has been defined within the algebraic framework of complete lattice theory. In this framework, dealing with deterministic and increasing operators, a dilation (respectively an erosion) is an operation which is distributive over supremum (respectively infimum). From this simple definition of dilation and erosion, we cannot say much about the properties of them. However, when they form an adjunction, many important properties can be derived such as monotonicity, idempotence, and extensivity or anti-extensivity of their composition, preservation of infimum and supremum, etc. Mathematical morphology has been first developed in the setting of sets, and then extended to other algebraic structures such as graphs, hypergraphs or simplicial complexes. For all these algebraic structures, erosion and dilation are usually based on structuring elements. The goal is then to match these structuring elements on given objects either to dilate or erode them. One of the advantages of defining erosion and dilation based on structuring elements is that these operations are adjoint. Based on this observation, this paper proposes to define, at the abstract level of category theory, erosion and dilation based on structuring elements. We then define the notion of morpho-category on which erosion and dilation are defined. We then show that topos and more precisely topos of presheaves are good candidates to generate morpho-categories. However, topos do not allow taking into account the notion of inclusion between substructures but rather are defined by monics up to domain isomorphism. Therefore we define the notion of morpholizable category which allows generating morpho-categories where substructures are defined along inclusion morphisms. {A direct application of this framework is to generalize modal morpho-logic to other algebraic structures than simple sets.}
\end{abstract}


%================================ Start Text

%%

\section{Introduction}

Mathematical morphology~\cite{BHR07,IGMI_NajTal10} is a formal technique which allows processing information in a non-linear fashion. It has been created in the 60s by G. Matheron and J. Serra initially to be applied to mining problems. Since then, mathematical morphology was mainly applied in the field of image processing~\cite{Ser82,Ser88}. The first formalization of mathematical morphology was made within the framework of set theory, and more recently it has been extended to a larger family of algebraic structures such as graphs~\cite{CNS09,cousty2013,MS09,Vin89}, hypergraphs~\cite{BB13,IB:DAM-15} and simplicial complexes~\cite{DCN11}. The two basic operations of mathematical morphology are erosion and dilation. When these two operations are defined within the framework of set theory, they are often based on structuring elements. The goal is then to match these structuring elements on given objects either to dilate or erode them.    

%\medskip
A general definition of mathematical morphology (in the case of deterministic and increasing operators) has been defined within the algebraic framework of complete lattice theory~\cite{HR90}. In this framework, a dilation (respectively an erosion) is an operation which is distributive over supremum (respectively infimum). In the framework of set theory, the complete lattice which is classically considered is $(\mathcal{P}(S),\subseteq)$ where $\mathcal{P}(S)$ is the powerset of the set $S$. From this simple definition of dilation and erosion, we cannot say much about their properties. By contrast, when they form an adjunction, many important properties can be derived such as monotonicity, idempotence, preservation of infimum and supremum, etc. Now, what is observed in all extensions of mathematical morphology to other algebraic structures than simple sets such as graphs, hypergraphs, etc., is that the process is always the same: defining the complete lattice from which erosion and dilation based on structuring elements are defined. Structuring elements are typically defined from a binary relation between elements of the underlying space (on which the algebraic structure is built). One of the advantages of defining dilation and erosion based on structuring elements is that in this case they form an adjunction. 

In this paper, we propose to abstractly formalize (i.e. independently of a given algebraic structure) this approach to define the adequate complete lattice from which dilation and erosion based on structuring elements will be defined, and then to obtain directly the adjunction property as well as all the other algebraic properties mentioned previously. To this end, we propose to use the framework of category theory. We define a family of categories, called {\em morpho-categories} from which we abstractly define the notions of structuring element, erosion and dilation  such that these two mathematical morphology operations form an adjunction. We then detail a series of examples of algebraic structures from the categorical notion of elementary topos~\cite{BW85,Law72}, and more precisely the topos of presheaves. The usefulness of topos of presheaves  is that they allow us to unify in a same framework most of algebraic structures such as sets, graphs, and hypergraphs. Indeed, sets, graphs, and hypergraphs and many other algebraic structures can be defined as presheaves whose the category is known to be a complete and co-complete topos. Another important advantage of topos is that they assure us that given an object $c$ of the topos, the set of subobjects $Sub(c)$ is a morpho-category over which we can then define erosion and dilation based on structuring elements. 

Now, although topos allow us to define erosion and dilation such that they form an adjunction, this framework is not completely satisfactory. The problem is that subobjects in $Sub(c)$ are  not defined uniquely but up to an isomorphism, and then relations between subobjects are monics rather than standard inclusions. This leads us to consider a larger family of categories than topos, called {\em morpholizable categories}, which will also allow us to  generate morpho-categories as before with topos but whose each element will be a subobject in a standard way (i.e. with respect to inclusion). 

\medskip
The paper is organized as follows: Section~\ref{preliminaries} reviews some concepts, notations and terminology about topos and mathematical morphology. In Section~\ref{sec:morpho-categories} we propose to define abstractly the concept of morpho-categories that we then illustrate with some examples of algebraic structures defined by presheaves. In Section~\ref{structuring} we introduce the two mathematical morphology operators of erosion and dilation based on structuring elements within the framework of morpho-categories. Some results will be also presented about them such as the property that they form an adjunction. We will further give sufficient conditions for these two operations to be anti-extensive and extensive, and dual with respect to each other. In Section~\ref{morpholizable categories} we introduce the notion of morpholizable categories over which we abstractly define a notion of inclusion between objects and show how to build morpho-categories in this framework. {Finally, in Section~\ref{morpho-logic}, we apply the proposed abstract framework to define semantics of both modalities $\Box$ and $\Diamond$, which are standard in modal logic. Hence, we extend the work in~\cite{Bloch02} where semantics of the modalities are also defined as morphological erosion and dilation but within the restricted framework of set theory. By considering a larger family of algebraic structures (graphs, hypergraphs, presheaves, etc.) the modal logic thus defined, so-called {\em morpho-logic}, will be more intuitionistic than classic.}

\section{Preliminaries: topos and mathematical morphology}
\label{preliminaries}

 \subsection{Topos}
 \label{topos}
 
This paper relies on many terms and notations from the categorical theory of topos. The notions introduced here make then use of basic notions of category theory (category, functor, natural transformation, limits, colimits, etc.), which are not recalled here, but interested readers may refer to textbooks such as~\cite{BW90,McL71}. 
 
% \medskip
 The theory of topos has been defined by A. Grothendieck at the end of the 50's~\cite{Gro57}. Here, we will not present Grothendieck's definition of topos, but the more general one of elementary topos defined in the late 60's by Lawvere and Tierney~\cite{Law72}. The presentation will remain succinct and  readers wanting to go into details in topos may refer to~\cite{BW85} for instance. 
 %\\ The interest for us of 
Topos constitute
%is that they will be 
a first family of categories from which we can define morpho-categories.
 
% \medskip
\subsubsection{Definitions}

 \paragraph{Notations.} In the following, morphisms are denoted by the lowercase latin letters $f$, $g$, $h$..., and natural transformations by the lowercase greek letters $\alpha$, $\beta$, etc. Standarly, a natural transformation $\alpha$ between two functors $F,F' : \mathcal{C} \to \mathcal{C}'$ is denoted by $\alpha : F \Rightarrow F'$. We write $f : A \rightarrowtail B$ or more simply $A \rightarrowtail B$ to indicate that $f$ is a monic. Finally, $dom(f)$ and $cod(f)$ denote, respectively, the domain and the codomain of a given morphism $f$. 
 
 \begin{definition}[Exponentiable and exponential]
Let $\mathcal{C}$ be a category. An object $A \in \mathcal{C}$ is {\bf exponentiable} if for every object $B \in \mathcal{C}$ there exists an object $B^A \in \mathcal{C}$, called {\bf exponential}, and a morphism $ev : B^A \times A \to B$, called {\bf evaluation}, such that for every morphism $f : A \times C \to B$ there exists a unique  morphism $f^\flat$, called {\bf exponential transpose}, for which the diagram 

$$
\xymatrix{
 A \times B^A \ar[r]^{ev} & B \\
 A \times C \ar[u]^{id_A \times f^\flat} \ar[ru]_f
  }
$$

commutes. We say that the category has {\bf exponentials} if every object has its exponential.
\end{definition}
 
 \begin{definition}[Cartesian closed]
A category $\mathcal{C}$ is called {\bf Cartesian closed} if $\mathcal{C}$ has finite products and exponentials.
\end{definition}

\begin{definition}[Subobject classifier]
Let $\mathcal{C}$ be a category. A {\bf subobject classifier} for $\mathcal{C}$ is an object $\Omega \in |\mathcal{C}|$ together with a morphism $\top : \mathbb{1} \to \Omega$ ($\mathbb{1}$ is the terminal object of $\mathcal{C}$ if it exists), called the {\bf true arrow}, such that for each monic $m : B \rightarrowtail A$, there is a unique morphism $\chi_B : A \to \Omega$, called the {\bf characteristic arrow} of $B$, such that the diagram

$$
\xymatrix{
 B \ar[r]^{!} \ar[d]^m & \mathbb{1} \ar[d]_\top \\
 A \ar[r]^{\chi_B} & \Omega
 }
$$
is a pullback. The morphism $\top \circ !$ is often denoted $\top_B$. 
\end{definition}
%\Isa{La notation ! est classique ou il faut la rappeler ? Dire que c'est la generalisation de la fonction caracteristique ?}
%\Marc{La notation est classique. Je ne suis pas sur qu'il faille ajouter du commentaire.}

It is quite simple to show that subobject classifiers are unique up to an isomorphism. Moreover, given an object $A$ of a category $\mathcal{C}$ with a subobject classifier, we denote by $Sub(A)$ the set $\{[f] \mid cod(f) = A~\mbox{and}~\mbox{$f$ is a monic}\}$ where $[f]$ is the equivalence class of $f$ according to the equivalence relation $\approx$ defined by: 

$$
\begin{array}{lll}
f \approx g & \Longleftrightarrow & 
\xymatrix{
 C \ar[d]_{\simeq} \ar[dr]^f \\
 B \ar[r]_{g} & A
  }
  \end{array}
 $$
 where $B \simeq C$ means that $B$ and $C$ are isomorphic objects.
 Then, we have that $Sub(A) \cong Hom_\mathcal{C}(A,\Omega)$, where $\cong$ denotes isomorphism between morphisms, and $Hom_\mathcal{C}(A,\Omega)$ is the set of morphisms in $\mathcal{C}$ from $A$ into $\Omega$ (the mapping $[f] \in Sub(A) \mapsto \chi_{dom(f)} \in Hom_{\mathcal{C}}(A,\Omega)$ is obviously a bijection). This then leads to the definition of power object defined next.
 
 \begin{definition}[Powerobject]
 Let $\mathcal{C}$ be a Cartesian closed category with a subobject classifier $\Omega$. The {\bf power object} $P(A)$ of an object $A \in |\mathcal{C}|$ is defined as the object $\Omega^A$ (exponential of $A$ and $\Omega$).  
 \end{definition}
 
 \begin{definition}[Topos]
 A {\bf topos} is a Cartesian closed category $\mathcal{C}$ with finite limits (i.e. limits for all finite diagrams) and a subobject classifier. 
 \end{definition}
 
When $\mathcal{C}$ is a topos, we have that $Hom_\mathcal{C}(A \times B,\Omega) \cong Hom_\mathcal{C}(A,\Omega^B)$ ($\mathcal{C}$ is Cartesian closed), and then $Sub(A \times B) \cong Hom_\mathcal{C}(A,P(B))$. More precisely, let us denote by $Hom_\mathcal{C}(\_,P(B)) : \mathcal{C}^{op} \to Set$ the standard contravariant hom functor, and  by $Sub(\_ \times B) : \mathcal{C} \to Set$ the contravariant functor which associates to any object $A \in |\mathcal{C}|$ the set $Sub(A \times B)$ and to any morphism $h : A \to A'$ the mapping $Sub(h \times B): Sub(A' \times B) \to Sub(A \times B)$ defined by $[f'] \mapsto [f]$ for which there exists a morphism $g : dom(f) \to dom(f')$ such that $dom(f)$ with $g$ and $f$ is a pullback of $f'$ along $h \times Id_B$. We then have that $Hom_\mathcal{C}(\_,P(B)) \cong Sub(\_ \times B)$ naturally, i.e. there is a natural transformation $\varphi(\_,B) : Hom_\mathcal{C}(\_,P(B)) \Rightarrow Sub(\_ \times B)$ such that for every $A \in \mathcal{C}$, $\varphi(A,B)$ is bijective~\footnote{This generalizes to topos the set property which to every binary relation $r \subseteq A \times B$ associates a unique mapping $f_r : A \to \mathcal{P}(B)$ with $f_r(a) = \{b \mid (a,b) \in r\}$.}. The subobject of $P(B) \times B$ corresponding to $Id_{P(B)}$ by $\varphi(P(B),B)$ is denoted $\ni_B$. 
 
 \medskip
 The following results are standard  for topos:
 
 \begin{itemize}
 \item Every topos has also finite colimits, and then it has an initial object and a final object which are respectively the limit and the colimit of the empty diagram. The initial object is denoted by $\emptyset$.
 \item Every morphism $f$ in a topos can be factorized uniquely as $m_f \circ e_f$ where $e_f$ is an epic and $m_f$ is a monic. The codomain of $e_f$ is often denoted by $Im(f)$, and then $A \stackrel{f}{\rightarrow} B = A \stackrel{e_f}{\rightarrow} Im(f) \stackrel{m_f}{\rightarrowtail} B$. 
 \end{itemize}
 (Readers interested by the proofs of these results may refer to~\cite{BW90}.)
 
 \subsubsection{Example: the category of presheaves $Set^{\mathcal{C}^{op}}$}
 \label{examples of preseheaves}
 
An interesting category for us in the following is the category of presheaves $Set^{\mathcal{C}^{op}}$. Let $\mathcal{C}$ be a small category. Let us denote by $Set^{\mathcal{C}^{op}}$ the category of contravariant functors $F : \mathcal{C}^{op} \to Set$ (presheaves) where $Set$ is the category of sets. When the category $\mathcal{C}$ is a small category (i.e. both collections of objects and arrows are sets), it is known that the category $Set^{\mathcal{C}^{op}}$ is complete and co-complete (i.e. it has all limits and colimits). Then, let us show that it is a topos. 

\paragraph{$Set^{\mathcal{C}^{op}}$ is Cartesian closed.} The product of two functors $F ,G : \mathcal{C}^{op} \to Set$ is the functor $H : \mathcal{C}^{op} \to Set$ defined for every $C \in \mathcal{C}$ by $H(C) = F(C) \times G(C)$, and for every $f : A \to B \in \mathcal{C}$ by the mapping $H(f) : H(B) \to H(A)$ defined by $\{(a,b) \in H(A)  \mid (f(a),f(b)) \in H(B)\}$. \\ To define the exponential of two functors $F,G: \mathcal{C}^{op} \to Set$, we define the object $G^F$ as the functor which associates to any object $C \in |\mathcal{C}|$ the set of natural transformations from $Hom(\_,C) \times F$ to $G$. For every $f : A \to B \in \mathcal{C}$, $G^F(f) : G^F(B) \to G^F(A)$ is the mapping which associates to any natural transformation $\alpha : Hom(\_,B) \times F \Rightarrow G$ the natural transformation $\beta : Hom(\_,A) \times F \Rightarrow G$ defined for every object $C \in |\mathcal{C}|$ by $\beta_C(g : C \to A,c \in F(C)) = \alpha_C(f \circ g,c)$.  

\paragraph{$Set^{\mathcal{C}^{op}}$ has a subobject classifier.} In order to identify the subobject $\Omega$, we should first identify the subobjects of presheaves. Note that the action of a monic $i : F \Rightarrow G$ in $Set^{\mathcal{C}^{op}}$ ($i$ is a natural transformation) on any $A \in |\mathcal{C}|$ results in an injective mapping $i_A : F(A) \to G(A)$, and then $F$ is a sub-presheaf of $G$ if for every object $A \in |\mathcal{C}|$, $i_A$ is the set-inclusion. Then, for every $A \in |\mathcal{C}|$ let us denote by $Sieve(A)$ the set of all sieves on $A$ where a sieve is a set $S$ of arrows $f$ in $\mathcal{C}$ such that:

\begin{enumerate}
\item If $f \in S$ then $cod(f) = A$, and 
\item If $f \in S$ and $cod(g) = dom(f)$, then $f \circ g \in S$.
\end{enumerate}

Given a monic $i : F \Rightarrow G$ in $Set^{\mathcal{C}^{op}}$, let us define for every $A \in |\mathcal{C}|$ the mapping $\chi_{i(A)} : G(A) \to Sieve(A)$ by: $\forall x \in G(A)$

$$\chi_{i(A)}(x) = 
\left\{
\begin{array}{ll}
\{f : B \to A \mid G(f)(x) \in i_B(F(B))\} & \mbox{if $x \in i_A(F(A))$} \\
\emptyset & \mbox{otherwise}
\end{array}
\right.$$ 
So, if we define the functor $\Omega : \mathcal{C}^{op} \to Set$ by associating to any $A \in |\mathcal{C}|$ the set $Sieve(A)$, then we have respectively

\begin{itemize}
\item the natural transformation $\chi_i : G \Rightarrow \Omega$ which to every $A \in |\mathcal{C}|$ and to every $x \in G(A)$, sets $\chi_{i_A}(x) =  \chi_{i(A)}(x)$;
\item the natural transformation $T : \mathbb{1} \Rightarrow \Omega$ which~\footnote{$\mathbb{1} : \mathcal{C} \to Set$ is the presheaf which associates to any $A \in |\mathcal{C}|$ the terminal object $\mathbb{1}$ in $Set$.} to every $A \in |\mathcal{C}|$, $T_A$ associates to the unique element in $\mathbb{1}(A)$ the maximal sieve on $A$ (i.e. the set $\Omega(A)$).
\end{itemize}

These natural transformations $\chi_i$ and $T$ make the diagram

$$
\xymatrix{
 F \ar@{=>}[r]^{!} \ar@{=>}[d]^i & \mathbb{1} \ar@{=>}[d]_\top \\
 G \ar@{=>}[r]^{\chi_i} & \Omega
 }
$$
be a pullback square for every $A \in \mathcal{C}$, and then $\Omega$ is a subobject classifier in $Set^{\mathcal{C}^{op}}$. 

\paragraph{Examples of categories $\mathcal{C}$.}  Sets, undirected graphs, directed graphs, rooted trees and hypergraphs can be defined as presheaves.  Indeed, for each of them, the (base) category $\mathcal{C}$ can be graphically defined as:
 $$
 \begin{array}{lllllllll}
 \xymatrix{
 \bullet   \ar@(ul,ur) 
 } & \ &
 \xymatrix{
 V \ar@/^/[r]^{s}\ar@/_/[r]_{t} &  E \ar@(ul,ur)^{s(t),t(s)} \\
  } & \ &
 \xymatrix{
 V \ar@/^/[r]^s\ar@/_/[r]_t &  E
  } & \ &
  (\mathbb{N},\leq) & \ &
  V \rightarrow R \leftarrow E
 \end{array}$$
 %\Isa{Faut-il expliquer les notations dans ces graphiques ?} \\
 %\Marc{Ce sont des topos tr\`es classiques.}
 
 The category $\mathcal{C}$ for undirected graphs satisfies in addition the two equations: $s \circ s(t),t(s) = t$ and $t \circ s(t),t(s) = s$.
 
 \medskip
 For each case, the subobject classifier can be refined. Indeed, let us consider the case of directed graphs. Then, let $G : \mathcal{C}^{op} \to Set$ be a directed graph. We have seen that $\Omega(V)$ and $\Omega(E)$ are respectively $Sieve(V)$ and $Sieve(E)$, i.e. $\Omega(V) = \{\emptyset,\{id_V\}\}$ and $\Omega(E) = \{\emptyset,\{s\},\{t\},\{s,t\},\{s,t,id_E\}\}$. For any presheaf category $Set^{\mathcal{C}^{op}}$, given a morphism $f : X \to Y$ in  $\mathcal{C}$, there is a straightforward way to compute $\Omega(f) : \Omega(Y) \to \Omega(X)$. Indeed, for any sieve $S$ on $Y$ we have 
 $$\Omega(f)(S) = \{g : Z \to X \mid Z \in \mathcal{C}~\mbox{and}~f \circ g \in S\}$$
 
 Hence, we have that 
 $$\begin{array}{lll}
 \Omega(s)(\emptyset) = \emptyset & \ \ \ \  & \Omega(t)(\emptyset) = \emptyset \\
 \Omega(s)(\{s\}) = \{Id_V\} & \ \ \ \  & \Omega(t)(\{s\}) = \emptyset \\
 \Omega(s)(\{t\}) = \emptyset & \ \ \ \  & \Omega(t)(\{t\}) = \{Id_V\} \\
 \Omega(s)(\{s,t\}) = \{Id_V\} & \ \ \ \  & \Omega(t)(\{s,t\}) = \{Id_V\} \\
 \Omega(s)(\{s,t,Id_E\}) = \{Id_V\} & \ \ \ \  & \Omega(t)(\{s,t,Id_E\}) = \{Id_V\} \\
 \end{array}$$
 Therefore, the subobject classifier is the directed graph: 
 $$ 
 \xymatrix{
 F \ar@(ul,ur)^{\emptyset} \ar@/^/[r]^t  &  T  \ar@/^/[l]^s \ar@(ul,ur)^{\{s,t\}} \ar@(dl,dr)_{\{s,t,id_E\}} \\
  } 
 $$
 where we denote by $F$ the vertex $\emptyset$ and by $T$ the vertex $\{Id_V\}$.
 
 The true arrow $T : \mathbb{1} \to \Omega$ maps the unique vertex to $T$ and the unique arrow to $\{s,t,id_E\}$. 
 
 \subsubsection{Algebraization of subobjects}
 
 Let $\mathcal{C}$ be a topos. Let $X$ be an object of $\mathcal{C}$. Let us define the partial ordering $\preceq$ on $Sub(X)$ as follows: for all $f : A \rightarrowtail X$ and $g :  B \rightarrowtail X$
 $$[f] \preceq [g] \Longleftrightarrow \exists h : A \rightarrowtail B, f = g \circ h$$
 
 \begin{proposition}
 \label{is a lattice}
 $(Sub(X),\preceq)$ is a bounded lattice. It is a complete lattice if $\mathcal{C}$ is a complete topos (i.e. $\mathcal{C}$ has all limits and colimits).
 \end{proposition}
 
 \begin{proof}
 Obviously, $\preceq$ is reflexive and transitive. To show anti-symmetry, let us suppose that $[f] \preceq [g]$ and $[g) \preceq [f]$. This means that there exist $h$ and $h'$ such that $f = g \circ h$ and $g = f \circ h'$ whence we have both that $f = f \circ h' \circ h$ and $g = g \circ h \circ h'$. We can deduce that $h' \circ h = Id_A$ and $h \circ h' = Id_B$, and then $A \simeq B$ that is $[f] = [g]$.  Given two elements $[f]$ and $[g]$ of $Sub(X)$ with $f : A \rightarrowtail X$ and $g : B \rightarrowtail X$, the infimum of $[f]$ and $[g]$ is $[A \cap B \rightarrowtail A \stackrel{f}{\rightarrowtail} X]$ which is equivalent to $[A \cap B \rightarrowtail B \stackrel{g}{\rightarrowtail} X]$ (we recall that every topos has pullbacks, and pullbacks of monics are monics), and the supremum of $[f]$ and $[g]$ is $[A\cup B \rightarrowtail X]$ where $A\cup B \rightarrowtail X$ is the unique morphism, consequence of the fact that $A \cup B$ is the pushout of $A \cap B$. This pushout exists because pushouts of monics are monics in toposes. %finies ou pas forcement ?} \Marc{Pas forc\'ement, le topos est consid\'er\'e comme complet, et il a donc toutes les limites}
 Hence,  $(Sub(X),\preceq)$ is a lattice. Finally, it is bounded because $\mathcal{C}$ is a topos, and then $[\emptyset \rightarrowtail X]$ (in any topos, the unique morphism from the initial object $\emptyset$ to any other object is always a monic) and $[id_X]$ belongs to $Sub(X)$. These definitions of supremum and infimum extent to any family of elements of $Sub(X)$ when $\mathcal{C}$ is complete, and then, if $\mathcal{C}$ is complete, then so is the lattice $(Sub(X),\preceq)$. 
 \end{proof}
 
 Actually, $(Sub(X),\preceq)$ satisfies a stronger result, it is a Heyting algebra. This has generated many works (initiated by Lawvere and Tierney~\cite{Law72}) which establish strong connections between elementary toposes and intuitionistic logics. 
 
 \begin{proposition}
 \label{is an heyting algebra}
 $(Sub(X),\preceq)$ is a Heyting algebra.
 \end{proposition}
 
 \begin{proof}
 Let us show first that $Sub(\mathbb{1})$ where $\mathbb{1}$ is the terminal object of the topos $\mathcal{C}$ is a Heyting algebra, i.e. categorically $Sub(\mathbb{1})$ is finitely complete and finitely co-complete, and it is Cartesian closed. We have seen in Proposition~\ref{is a lattice} that $Sub(\mathbb{1})$ is both finitely complete and finitely co-complete. Finite product is defined by the infimum of equivalence classes of monics introduced in Proposition~\ref{is a lattice}. It remains to show that $Sub(\mathbb{1})$ has exponentials. Let $[U \rightarrowtail \mathbb{1}]$ and $[V \rightarrowtail \mathbb{1}]$ be two subobjects of $\mathbb{1}$. As $\mathcal{C}$ is a topos, $U$ and $V$ have an exponential $V^U$. As $\mathbb{1}$ is terminal in $\mathcal{C}$, there is a unique morphism $v^u : V^U \to \mathbb{1}$. It remains to show that the morphism $v^u$ is a monic. Let $f,g : S \to V^U$ be two morphisms. Under the (exponential) bijection, we have that $Hom_\mathcal{C}(S,V^U) \simeq Hom_\mathcal{C}(S \times U,V)$, and then we have for $f$ and $g$ corresponding morphisms $\overline{f},\overline{g} : S \times U \to V$. Let us suppose that $v^u \circ f = v^u \circ g$. This means that  $\overline{f} \circ u = \overline{g} \circ u$. But $u$ is a monic and then $\overline{f} = \overline{g}$ which implies by bijection that $f = g$. We then deduce that $v^u$ is a monic, and we conclude that $Sub(\mathbb{1})$ is a Heyting algebra. \\ It is not difficult to show that $Sub(X) \simeq Sub_{\mathcal{C}_{/X}}(Id_X)$ where $\mathcal{C}_{/X}$ is the slice category over $X$ and $Sub_{\mathcal{C}_{/X}}(Id_X)$ is the set of subobjects of the object $Id_X$ in the slice category $\mathcal{C}_{/X}$. By the fundamental theorem of topos theory~\footnote{This fundamental theorem states that if $\mathcal{C}$ is a topos, then for every object $X \in |\mathcal{C}|$, $\mathcal{C}_{/X}$ is a topos (see~\cite{BW85} for a proof of this result).} $\mathcal{C}_{/X}$ is also a topos, and then $Sub_{\mathcal{C}_{/X}}(Id_X)$ is a Heyting algebra ($Id_X$ is terminal in $\mathcal{C}_{/X}$), whence we conclude under the isomorphism $Sub(X) \simeq Sub_{\mathcal{C}_{/X}}(Id_X)$ that $Sub(X)$ is a Heyting algebra.    
 \end{proof}
 
 \subsection{Mathematical morphology on sets}
 
 The most abstract way to define dilation and erosion is as follows. Let $(L, \preceq)$ and $(L', \preceq')$ be two (complete) lattices. Let $\vee$ and $\vee'$ denote the supremum in $L$ and in $L'$, associated with $\preceq$ and $\preceq'$, respectively. Similarly, let $\wedge$ and $\wedge'$ denote the infimum in $L$ and in $L'$, respectively. An algebraic dilation is an operator $\delta : L \to L'$  that commutes with the supremum, i.e.
\[
\forall (a_i)_{i \in I} \in L, \delta(\vee_{i \in I} a_i) = \vee'_{i \in I} \delta(a_i)
\]
where $I$ denotes any index set (not fixed). 
An algebraic erosion is an operator $\varepsilon : L' \to L$ that commutes with the infimum, i.e.
\[
\forall (a_i)_{i \in I} \in L', \varepsilon(\wedge'_{i \in I} a_i) = \wedge_{i \in I} \varepsilon(a_i)
\]
where $I$ denotes any index set (not fixed). 
It follows that $\delta$ preserves the least element $\bot$ in $L$ ($\delta(\bot) = \bot$), and $\varepsilon$ preserves the greatest element $\top'$ in $L'$ ($\varepsilon(\top')=\top'$). 

Now, in practice, morphological operators are often defined on sets (i.e. $L$ and $L'$ are the powersets or finite powersets of given sets $S$ and $S'$, and often $S=S'$ and $L=L'$) through a structuring element designed in advance. Mathematical morphology has been mainly applied in image processing. In this particular case, the set $S$ is an Abelian group equipped with an additive internal law $+$, and its elements represent image points. Let us recall here the basic definitions of dilation and erosion $\delta_B$ and $\varepsilon_B$ in this particular case, where $B$ is a set called structuring element, under an additional hypothesis of invariance under translation. Let $X$ and $B$ be two subsets of $S$. The dilation and erosion of $X$ by the structuring element $B$, denoted respectively by $\delta_B(X)$ and $\varepsilon_B(X)$, are defined as follows:
\begin{center}
$\delta_B(X) = \{x \in S \mid \check{B}_x \cap X \neq \emptyset\}$ \\
$\varepsilon_B(X) = \{x \in S \mid B_x \subseteq X\}$
\end{center}
where 
$B_x = \{x + b \in S \mid b \in B\}$ where $+$ is the additive law associated with $S$ (e.g. translation), and $\check{B}$ is the symmetrical of $B$ with respect to the origin of space. An example is given in Figure~\ref{fig:cell} in $\mathbb{Z}^2$ for a binary image. This simple example illustrates the intuitive meaning of dilation (expanding the white objects according to the size and shape of the structuring element) and erosion (reducing the white objects according to the structuring element).

\begin{figure}[htbp]
\centerline{
 \includegraphics[width=4cm]{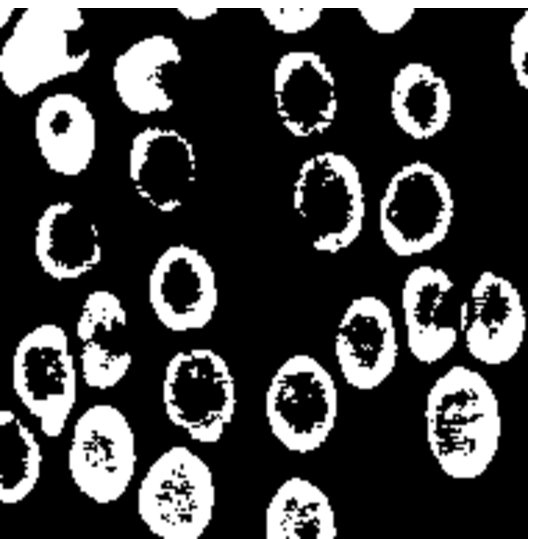} \hspace*{.1cm}
 \includegraphics[width=1cm]{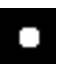} \hspace*{.1cm}
 \includegraphics[width=4cm]{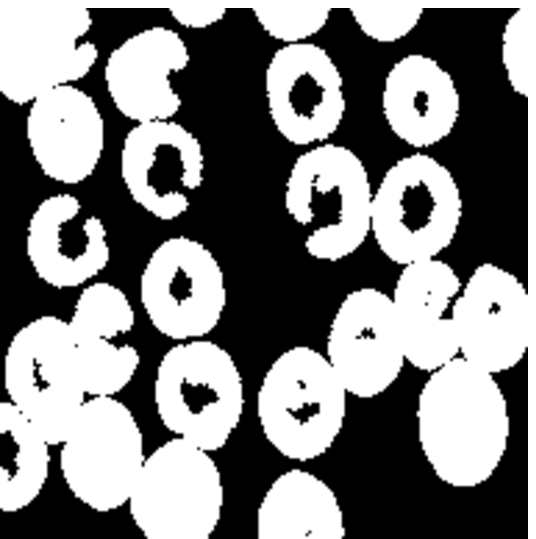} \hspace*{.1cm}
 \includegraphics[width=4cm]{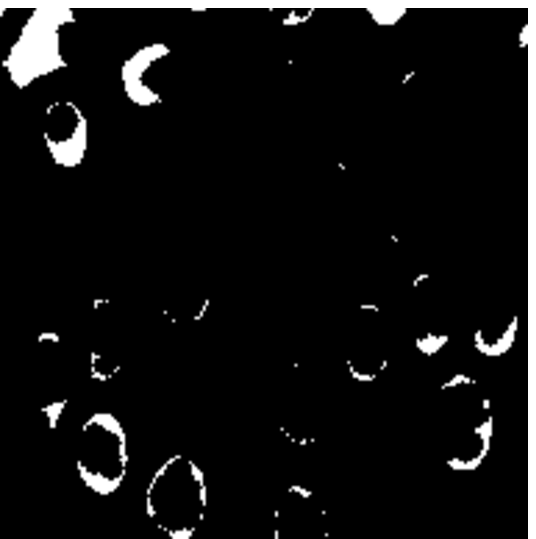} 
}
\caption{From left to right: binary image, structuring element, dilation, erosion.}
\label{fig:cell}
\end{figure}

Now, in a more generally setting, we consider any set $S$, not necessarily endowed with a particular structure, and the structuring element $B$ can then be seen as a binary relation $R_B$ on the set $S$ ($R_B \subseteq S \times S$) as follows: $(x,y) \in R_B \Longleftrightarrow y \in B_x$~\cite{BHR07,madrid2019}.\footnote{In the particular example of a set with an additive law $+$, the corresponding relation would be $(x,y) \in R_B$ iff $\exists b \in B, y = x+b$.} This is the way we will consider structuring elements in this paper, as done in previous work, in particular for mathematical morphology on graphs (see e.g.~\cite{IB:DAM-15,cousty2013,MS09}, among others) or logics (see e.g.~\cite{MA:AIJ-18,MA:IJAR-18,AB19,BHR07,Bloch02,BL02,Gorogiannis2008a}).

The most important properties of dilation and erosion based on a structuring element are the following ones~\cite{BHR07,IGMI_NajTal10,Ser82}:
\begin{itemize}
\item {\bf Monotonicity:} if $X \subseteq Y$, then $\delta_B(X) \subseteq \delta_B(Y)$ and $\varepsilon_B(X) \subseteq \varepsilon_B(Y)$; if $B \subseteq B'$, then $\delta_B(X) \subseteq \delta_{B'}(X)$ and $\varepsilon_{B'}(X) \subseteq \varepsilon_B(X)$. 
\item If for every $x \in E$, $x \in B_x$ (and this condition is actually necessary and sufficient), then 
\begin{itemize}
\item {\bf $\delta_B$ is extensive:} $X \subseteq \delta_B(X)$;
\item {\bf $\varepsilon_B$ is anti-extensive:} $\varepsilon_B(X) \subseteq X$.
\end{itemize}
\item {\bf Commutativity:} $\delta_B(X \cup Y) = \delta_B(X) \cup \delta_B(Y)$ and $\varepsilon_B(X \cap Y) = \varepsilon_B(X) \cap \varepsilon_B(Y)$ (and similar expressions for infinite or empty families of subsets).  
\item {\bf Adjunction:} $X \subseteq \varepsilon_B(Y) \Leftrightarrow \delta_B(X) \subseteq Y$.
\item {\bf Duality:} $\varepsilon_B(X) = [\delta_{\check{B}}(X^C)]^C$ where $\_^C$ is the set-theoretical complementation. 
\end{itemize}
Hence, $\delta_B$ and $\varepsilon_B$ are particular cases of general algebraic dilation and erosion on the lattice $(\mathcal{P}(S),\subseteq)$.

\section{Morpho-categories}
\label{sec:morpho-categories}

Morpho-categories will be categories on which we will define both basic mathematical morphology operations of erosion and dilation based on a structuring element.

%\medskip
In the following, we will denote objects of categories by the lowercase latin letters $c$, $d$ and $e$.

\subsection{Definition}

\begin{definition}[Morpho-category]
\label{morpho-category}
A {\bf morpho-category} $\mathcal{C}$ is a category together with a functor $U : \mathcal{C} \to Set$, which is further complete and co-complete, and such that:

\begin{itemize}
\item morphisms in $\mathcal{C}$ define a partial order $\rightarrowtail$,\footnote{We choose the notation $\rightarrowtail$ to denote the partial order because necessarily morphisms in $\mathcal{C}$  are monics as for all objets $c,c' \in |\mathcal{C}|$, there is at most one morphism in $Hom_\mathcal{C}(c,c')$.}
\item $U$ preserves monics\footnote{and then for every $f \in Hom_\mathcal{C}$, $U(f)$ is an injective mapping.} and colimits.
\end{itemize}
\end{definition} 

\begin{notation}
$\mathcal{C}$ formulates internally\footnote{By ``internally'', we mean that the traditional definition of complete lattice which takes place in sets is defined here to other categorical structures.} the notion of complete lattice. We will then use the standard notations of complete lattices, and denote the limit and colimit of diagrams $D$ by $\bigwedge D$ and $\bigvee D$, respectively. 
\end{notation}

The archetypical example of morpho-category is the complete lattice $(\mathcal{P}(S),\subseteq)$ for a given set $S$ together with the forgetful functor identity on $\mathcal{P}(S)$. In the next section, we will present more examples of morpho-categories, all of them generated from presheaf toposes. 

\medskip
As $\mathcal{C}$ is complete and co-complete, it has initial and terminal objects which are, respectively, the limit and the colimit of the empty diagram. The initial object is denoted by $\emptyset$, and the terminal one by $t$.  

\begin{remark}
In all the examples that we will develop in this paper, the functor $U$ is faithful, and then $\mathcal{C}$ is a concrete category. However, this property is not useful for us to define our two operations of erosion and dilation with good properties.  %\Isa{tu veux dire que meme si la propriete est vraie on n'en a pas vraiment besoin dans les preuves ou autres ?} \Marc{Oui}
\end{remark}

\begin{remark}
\label{remark:expoential}
{Given a morpho-category $\mathcal{C}$, we can define for all objects $c,d \in |\mathcal{C}|$, the object $d^c$ as the supremum of the diagram $\{e \in |\mathcal{C}| \mid e \wedge c \rightarrowtail d\}$. It is known that the object $d^c$ is the exponential of $d$ and $c$ if $\mathcal{C}$satisfies the infinite distributive law:
$$c \wedge \bigvee \mathcal{D} = \bigvee \{c \wedge e \mid e \in |\mathcal{D}|\}$$
where $\mathcal{D}$ is a diagram. Indeed, we have that $c \wedge \bigvee\{e \in |\mathcal{C}| \mid e \wedge c \rightarrowtail d\} = \bigvee \{e \wedge c \mid e \wedge c \rightarrowtail d\}$, and then we can conclude that $d^c \wedge c \rightarrowtail d$. The existence of the exponential transpose is obvious.}
\end{remark}

\subsection{Presheaf Toposes as generators of morpho-categories}
\label{Ex:morpho-categories}

In any complete topos $\mathcal{C}$, given an object $c \in \mathcal{C}$,  we have seen that $Sub(c)$ is a complete lattice (see Proposition~\ref{is a lattice}). So, in the case of a presheaf topos whose base category is $\mathcal{B}$, given a presheaf $F : \mathcal{B}^{op} \to Set$, each $b \in |\mathcal{B}|$ gives rise to a  forgetful functor $U_b : [\alpha] \mapsto dom(\alpha)(b)$ where $[\alpha] \in Sub(F)$. By definition of the order $\preceq$ on $Sub(c)$, this functor preserves monics. Let us then show that it preserves colimits. First, let us remark that for every equivalence class $[\alpha] \in Sub(F)$ there exists the presheaf $F_{[\alpha]} : \mathcal{B}^{op} \to Set$ which to every $b' \in |\mathcal{B}|$ satisfies $F_{[\alpha]}(b') \subseteq F(b')$. Therefore, given a diagram $\mathcal{D}$ in $Sub(c)$, the colimits of $\mathcal{D}$ is defined by the equivalence class $[\beta : F_{\hookrightarrow} \Rightarrow F]$ where $F_{\hookrightarrow} : \mathcal{B}^{op} \to F$ is the presheaf defined for every $b' \in |\mathcal{B}|$ by: 

$$F_{\hookrightarrow}(b') = \bigcup_{[\alpha] \in \mathcal{D}}F_{[\alpha]}(b')$$ 

By construction,  $[\beta]$ is the colimit of $\mathcal{D}$ in $Sub(F)$, and then we can deduce that $U_b$ preserves colimits. By following the same process, we can also show that $U_b$ preserves limits.

\medskip
Hence, for graphs and hypergraphs, we can define two kinds of forgetful functors, one which forgets on vertices and the other which forgets on edge and hyperedges~\footnote{Due to the definition of hypergraphs by presheaves, we can also forget about the set associated with the object $R$ of the base category (see Section~\ref{examples of preseheaves}). This forgetful functor is actually similar to the one on hyperedges.}. In presheaf topos of sets, we can associate a unique kind of forgetful functor $U : Sub(F) \to Set$ which, given a subobject $[\alpha]$, yields a set $S'$ such that the functor $F' : \bullet \mapsto S'$ is isomorphic to $dom(\alpha)$.

\subsection{Invariant properties}

Here, we state some properties on morpho-categories which will be useful for establishing some of our results in Section~\ref{structuring}. The adjective ``invariant'' means that these properties are invariant by equivalence of categories~\cite{CL16}. 

\medskip
Let $\mathcal{C}$ be a morpho-category whose forgetful functor is $U : \mathcal{C} \to Set$.

\begin{itemize}
\item {\bf Atomicity.} An {\bf atom} for $U$ is an object $c \in \mathcal{C}$ such that $U(c)$ is a singleton, and for every other object $c' \in |\mathcal{C}|$ with $U(c') = U(c)$, we have that $c \rightarrowtail c'$. \\ $\mathcal{C}$ {\bf has atoms for} $U$ if for every $x \in U(t)$ ($t$ is terminal in $\mathcal{C}$), there exists an atom $c_x \in |\mathcal{C}|$ with $U(c_x) = \{x\}$. \\ $\mathcal{C}$ is said {\bf atomic for} $U$ if any object $c \in |\mathcal{C}|$ is sup-generated from atoms, i.e. $c$ is the colimit of some diagrams of atoms\footnote{This notion of atomicity can be compared with the notion of atom for toposes (see for instance~\cite{CL16}).}. 
\item {\bf Boolean.} $\mathcal{C}$ is said {\bf Boolean} if $\mathcal{C}$ is Cartesian closed~\footnote{We have seen in Remark~\ref{remark:expoential} that in this case, for all the objects $c,d \in |\mathcal{C}|$, the exponential $d^c$ is given as the supremum of the diagram $\{e \in |\mathcal{C}| \mid e \wedge c \rightarrowtail d\}$.} and for every object $c \in |\mathcal{C}|$, we have further $c \vee \overline{c} = t$ where $\overline{c} = \emptyset^c$ (the exponential of the object $c$ and the initial object $\emptyset$ of $\mathcal{C}$). 
\end{itemize}

It is obvious from definitions that these properties are invariant by equivalence of categories. 

\medskip
All the examples developed in this paper have atoms for all the forgetful functors which have been defined on them. For the morpho-categories over the topos of sets, atoms are singletons, and all of them are atomic for any forgetful functor $U$. For the topos of graphs, if the forgetful functor $U$ forgets on vertices, then atoms are all the graphs with only one vertex and no edge. Now, if the forgetful functor $U$ forgets on edges, then atoms are all the graphs with only one edge $e$, and the source and target of $e$ as vertices. Morpho-categories over the topos of graphs are only atomic for the forgetful functors which forget on edges. 
Atoms for hypergraphs are similarly defined. 

In the same way, all examples developed in this paper are Cartesian closed  morpho-categories because for toposes, given an object $X$, $Sub(X)$ is a Heyting algebra (see Proposition~ \ref{is an heyting algebra}). However, only morpho-categories over sets are Boolean. 

\begin{proposition}
Let $\mathcal{C}$ be a morpho-category which has atoms. Then, for every $x \in U(t)$, $c_x$ is unique.
\end{proposition}

\begin{proof}
Suppose that there exists another atom $c'_x$  with $U(c'_x) = \{x\}$. Then, we have both $c_x \rightarrowtail c'_x$ and $c'_x \rightarrowtail c_x$, and we can conclude by anti-symmetry that $c_x = c'_x$. 
\end{proof}

As Cartesian closed morpho-categories are internal Heyting algebras, they satisfy a number of properties such as distributivity, commutativity, order-reversing of complementation, etc. Here, we give and prove some of them that will be used in the following.

\begin{proposition}
Let $\mathcal{C}$ be a Boolean morpho-category. Then we have for every object $c,c' \in |\mathcal{C}|$: 

\begin{enumerate}
\item $c \wedge \overline{c} = \emptyset$, 
\item for every $c' \in |\mathcal{C}|$ such that $c \wedge c' = \emptyset$, $c' \rightarrowtail \overline{c}$,
\item $c \rightarrowtail c' \Longrightarrow \overline{c'} \rightarrowtail \overline{c}$, and
\item $\overline{\overline{c}} = c$.
\end{enumerate}
\end{proposition}

\begin{proof}
The first property is a direct consequence of the evaluation arrow. Indeed, we have that $\overline{c} \wedge c \rightarrowtail \emptyset$, and then $\overline{c} \wedge c = \emptyset$ ($\emptyset$ is initial in $\mathcal{C}$). \\ The second point is a consequence of distributive laws. Indeed, as $\mathcal{C}$ is Boolean, it is distributive, we have that $c' = c' \wedge (c \vee \overline{c}) = (c \wedge c') \vee (c' \wedge \overline{c}) = c' \wedge \overline{c}$ (the last equation comes from the fact that $c \wedge c' = \emptyset$). %\Isa{est-ce qu'il ne faut pas remplacer $c_1$ par $c'$ dans cette preuve pour avoir les memes notations que dans la propriete 2 ?} 
\\ To show the third point, let us suppose that $c \rightarrowtail c'$. Let us show that $c \wedge \overline{c'} = \emptyset$. Let us suppose the opposite, i.e. let us suppose that there exists $c'' \neq \emptyset$ such that $c \wedge \overline{c'} = c''$. We then have that both $c'' \rightarrowtail c'$ and $c'' \rightarrowtail \overline{c'}$ which is not possible. Hence, we have that $c \wedge \overline{c'} = \emptyset$ which by the second property leads to $\overline{c'} \rightarrowtail \overline{c}$. \\ The last point is also a simple consequence of distributive laws. Let us first show that $\overline{\overline{c}} \rightarrowtail c$. As $\mathcal{C}$ is distributive, we have that $\overline{\overline{c}} = \overline{\overline{c}} \wedge (c \vee \overline{c}) = (\overline{\overline{c}} \wedge c) \vee (\overline{\overline{c}} \wedge \overline{c}) = \overline{\overline{c}} \wedge c$ (the last equation comes from the fact that $\overline{\overline{c}} \wedge \overline{c}  = \emptyset$).  The proof to show that $c \rightarrowtail \overline{\overline{c}}$ is similar.  
\end{proof}
\section{Mathematical morphology over structuring element}
\label{structuring}

Here, we propose to generalize the definitions of erosion and dilation based on structuring elements in our categorical setting of morpho-categories. Taking inspiration from classical mathematical morphology where one probes the image with ``simple'' geometrical shapes called structuring elements, we propose to consider a structuring element as an object of the morpho-category $\mathcal{C}$ that we will use to construct morphological operators by ``matching'' this structuring object at different ``positions'' with more ``complex'' objects of $\mathcal{C}$.

\subsection{Erosion and dilation in morpho-categories}

Let $\mathcal{C}$ be a morpho-category, and $U$ its associated forgetful functor. As the morphisms between objects in $\mathcal{C}$ are unique, these are monics. Therefore, according to our conventions, given two objects $c,c' \in |\mathcal{C}|$, we will denote in the following $c \rightarrowtail c'$ the unique morphism between $c$ and $c'$ when it exists. Moreover, in the following, except when it will be explicitly stated, every set $\{c \in |\mathcal{C}| \mid P(c)\}$ where $P$ is a property resting on $|\mathcal{C}|$, will denote the diagram whose the objects are all the elements $c$ and morphisms are all the monics in $\mathcal{C}$ between them. Finally, given an object $c \in |\mathcal{C}|$ and an element $x \in U(c)$, we will denote by $x_t$ the element of $U(t)$ (we recall that $t$ is the terminal object of $\mathcal{C}$) such that $x_t = U(c \rightarrowtail t)(x)$. 

\begin{definition}[Structuring element]
A {\bf structuring element} is a mapping $b : U(t) \to |\mathcal{C}|$. 
\end{definition} 

\begin{example}
\label{structuring element set}
In the morpho-category $(\mathcal{P}(S),\subseteq)$, given a set $S$, any mapping $b : S \to \mathcal{P}(S)$ is a structuring element. This structuring element  associates to an element $x \in S$ a set $b(x)$ whose elements are in relation with $x$. \\ When mathematical morphology is applied in image processing, we saw that in this particular case, $S$ is an Abelian group equipped with an additive law $+$, and then given a point $x \in S$, $b(x) = \{x + b \mid b \in B\}$ where $B \subseteq S$. 
\end{example}

\begin{example}
\label{structuring element in toposes}
Let us suppose a complete and co-complete topos $\mathcal{C}$ together with a forgetful functor $U : \mathcal{C} \to Set$ satisfying the properties of Definition~\ref{morpho-category}. Given an object $c \in |\mathcal{C}|$, any mapping $b : U([Id_c]) \to |Sub(c)|$ is a structuring element. For instance, in the topos of directed graphs, i.e. the category of presheaves $F : \mathcal{B} \to Set$ where $\mathcal{B}$ is the base category~\footnote{Be careful, here $t$ designates the morphism between $V$ and $E$ in the base category $\mathcal{B}$, and should not be confused with the terminal object in $\mathcal{C}$.} 
 $$\xymatrix{
 V \ar@/^/[r]^s\ar@/_/[r]_t &  E \\
  }$$
%\Isa{pas le meme t ici ?} \Marc{Je ne suis pas s\^ur que ceci pose probl\`eme. A discuter.}

Given a graph $G : \mathcal{B}^{op} \to Set$, if we suppose that the forgetful functor $U$ forgets on vertices with $U([Id_G]) = G(V)$, we can consider the structuring element $b : G(V) \to Sub(G)$ which to any $x \in G(V)$ associates the subobject $[G_x \Rightarrow G]$ where $G_x : \mathcal{B}^{op} \to Set$ is defined by: 
\begin{itemize}
\item $G_x(V) = \{y \in G(V) \mid \exists e \in G(E), S_e = \{x,y\}\} \cup \{x\}$, and 
\item $G_x(E) = \{e \in G(E) \mid G(s)(e) = x~\mbox{or}~G(t)(e) = x\}$. 
\end{itemize}
where $S_e = \{G(s)(e),G(t)(e)\}$ for every edge $e \in G(E)$.
\end{example}

\subsubsection{Algebraic structures of structuring elements}

Let $StEl$ denote the set of all structuring elements of $\mathcal{C}$. A partial order on $StEl$ can be defined using pointwise injection: $b \preceq b' \Longleftrightarrow (\forall x \in U(t), b(x) \rightarrowtail b'(x))$. 

\begin{proposition}
$(StEl,\preceq)$ is a complete lattice.
\end{proposition}

\begin{proof}
$\mathcal{C}$ is a morpho-category. Hence, for all $(b_i)_{i \in I} \in StEl$, where $I$ denotes any index set, let us set
\begin{center}
$(\bigwedge_{i \in I} b_i)(x) = \bigwedge \{b_i(x) \mid i \in I\}$ \\
$(\bigvee_{i \in I} b_i)(x) = \bigvee \{b_i(x) \mid i \in I\}$
\end{center}
which are obviously the infimum (greatest lower bound) and the supremum (least upper bound) of the family of $b_i$, respectively.
The greatest element is the full structuring element defined by $full(x) = t$ (the terminal object in $\mathcal{C}$), and the least element is the empty structuring element defined by $emp(x) = \emptyset$ (the intial object in $\mathcal{C}$). 
\end{proof}

Likewise, an internal composition law $\star$ on $StEl$ can be defined. Let $b,b' : U(t) \to |\mathcal{C}|$ be two structuring elements. Let $b \star b' : U(t) \to |\mathcal{C}|$ be the structuring element defined for every $x \in U(t)$ by: 
$$(b \star b')(x) = \bigvee \{b'(y_t) \mid y \in U(b(x))\}$$ 
We show in the next proposition that the operation $\star$ is associative. To equip $StEl$ with an identity element, we need to impose that the category $\mathcal{C}$ has atoms for the forgetful functor $U$. Indeed, for morpho-categories $\mathcal{C}$ which have atoms for $U$, we can define the structuring element $sgt : x \mapsto c_x$, that will  be shown to be the identity element.% -- d'ailleurs on dit neutral element ou unit element ?} \Marc{Aucune id\'ee ;-)} . 

\begin{proposition}
\label{is a monoid}
Under the condition that $U$ reflects monics, $(StEl,\star,sgt)$ is a monoid.
\end{proposition}

\begin{proof}
Let us first show that $\star$ is associative, i.e. let us show that $(b \star b') \star b'' = b \star (b' \star b'')$. From the definition of the operation $\star$, we have that 

$$\begin{array}{ll}
((b \star b') \star b'')(x) & = \bigvee \{b''(z_t) \mid  z \in U(b \star b'(x))\} \\ 
& = \bigvee \{b''(z_t) \mid \exists y \in U(b(x)), z \in U(b'(y_t))\} \\
& = (b \star (b' \star b''))(x)
\end{array}$$  %\Isa{mais ca me semble encore trop rapide...}

\medskip
Let us now show that $sgt$ is the identity element for $\star$, i.e. for every $x \in U(t)$, we have $(b \star sgt)(x) = b(x)$ and $(sgt \star b)(x) = b(x)$. By definition of the operations $\star$ and $sgt$, we have that $(b \star sgt)(x) = \bigvee \{c_{y_t} \mid y \in U(b(x))\}$. Obviously, we have for every $y \in U(b(x))$ that $U(c_{y_t}) \rightarrowtail U(b(x))$, and then as $U$ reflects monics $c_{y_t} \rightarrowtail b(x)$. By the properties of colimits we can write that $(b \star sgt)(x) \rightarrowtail b(x)$. \\ Now, as $U$ preserves colimits, we also have that $U(b(x)) \rightarrowtail U((b \star sgt)(x))$, and then as $U$ reflects monics, we conclude that $b(x) \rightarrowtail (b \star sgt)(x)$. \\ We can conclude that $(b \star sgt)(x) = b(x)$. 

\medskip
By definition, we have that $(sgt \star b)(x) = \bigvee \{b(y_t) \mid y \in U(sgt(x))\}$, whence we can directly conclude that $(sgt \star b)(x) = b(x)$.
\end{proof}

All the examples developed in this paper (sets, graphs, hypergraphs) satisfy the conditions of Proposition~\ref{is a monoid} because each forgetful functor $U$ is faithful and then reflects monics and epics. 

\subsubsection{Erosion and dilation: Definitions}

\begin{definition}[Erosion]
\label{erosion}
Let $b$ be a structuring element of $\mathcal{C}$. The {\bf erosion over $b$}, denoted by $\varepsilon[b]$, is the mapping from $\mathcal{C}$ to $\mathcal{C}$ defined for every object $d \in |\mathcal{C}|$ as follows.  Let $D^\varepsilon_d$ be the diagram $\{e \in |\mathcal{C}| \mid \forall v \in U(e), b(v_t) \rightarrowtail d\}$, then 
$$\varepsilon[b](d) = \bigvee  D^\varepsilon_d$$
\end{definition}

\begin{example}
\label{Ex:erosion for sets}
In the morpho-category $(\mathcal{P}(S),\subseteq)$ for a given set $S$, for every $A \subseteq S$, we have $D^\varepsilon_A = \{E \subseteq S \mid \forall x \in E, b(x) \subseteq A\}$. Hence, we have that $\varepsilon[b](A) = \bigcup_{E \in D^\varepsilon_d} E = \{x \in S \mid b(x) \subseteq A\}$. \\ In the case where $S$ is an Abelian group with $+$ as additive law, given a structuring element $B$, we have that $D^\varepsilon_A = \{E \subseteq S \mid \forall x \in E, \{x+p \mid p \in B\} \subseteq A\}$, i.e. $\varepsilon[b](A) = \{x \in S \mid B_x \subseteq A\}$ where $B_x = \{x + p \mid p \in B\}$ ($= b(x)$). \\ This shows that Definition~\ref{erosion} includes the classical definition of $\varepsilon$ by $B$ on sets (see Section~\ref{preliminaries} for a reminder of mathematial morphology on sets). 
\end{example}

\begin{example}
\label{Ex:erosion for graphs}
In the morpho-category $(Sub(c),\preceq)$ for a given object $c$ in a complete topos $\mathcal{C}$ equipped with a forgetful functor $U$ satisfying all the conditions of Definition~\ref{morpho-category}, for every $[f] \in Sub(c)$, the set $D^\varepsilon_{[f]}$ is $D^\varepsilon_{[f]} = \{[g] \in Sub(c) \mid \forall x \in U([g]), b(x) \preceq [f]\}$, and then $\varepsilon[b]([f])$ is the colimit of $D^\varepsilon_{[f]}$. 

To illustrate this, let us take again the example of the topos of directed graphs with the forgetful functor $U$ defined on vertices and the mapping $b$ as in Example~\ref{structuring element in toposes}. Let $G : \mathcal{B}^{op} \to Set$ be a graph. Let $[\alpha]$ be a subobject of $Sub(G)$. In this case, $\varepsilon[b]([\alpha]) = [\beta]$ where $\beta : G' \Rightarrow G$ is the natural transformation with $G' : \mathcal{B}^{op} \to Set$ the graph defined by:

\begin{itemize}
\item $G'(V) = \{x \in G(V) \mid \forall e \in dom(\alpha)(E),x \in \alpha_V(S_e) \Rightarrow e \in \alpha^{-1}_E(G_x(E))\}$,\footnote{Let us recall that here $S_e$ is the set $\{dom(\alpha)(s)(e),dom(\alpha)(t)(e)\}$. We also recall by definition of natural transformation that $\alpha_V$ and $\alpha_E$ are the injective mapings $\alpha_V : dom(\alpha)(V) \to G(V)$ and $\alpha_E : doma(\alpha)(E) \to G(E)$ ($dom(\alpha)$ is a presheaf, i.e. $dom(\alpha) : \mathcal{B}^{op} \to Set$).} and
\item $G'(E) = \{e \in G(E) \mid G(s)(e),G(t)(e) \in G'(V)\}$.
\end{itemize}
Hence $\beta_V$ and $\beta_E$ are inclusion mappings.
\end{example}

\begin{proposition}
\label{prop:stability1}
We have both that $\emptyset \in D^\varepsilon_d$ (and then $\bigwedge D^\varepsilon_d = \emptyset$) and $\varepsilon[b](d) \in D^\varepsilon_d$.
\end{proposition}

\begin{proof}
As $U$ preserves colimits, it preserves initial object, and then this allows us to conclude that $\emptyset \in D^\varepsilon_d$. \\ In the same way, we know that $U$ preserves colimits. Consequently, we have for every $v \in U(\varepsilon[b](d))$ that $b(v_t) \rightarrowtail d$, whence we can conclude that $\varepsilon[b](d) \in D^\varepsilon_d$. 
\end{proof}

Hence, $\varepsilon[b](d)$ is the colimit of the cone $(\emptyset \rightarrowtail e)_{e \in D^\varepsilon_d}$,  and reciprocally $\emptyset$ is the limit of the cocone $(e \rightarrowtail \varepsilon[b](d))_{e \in D^\varepsilon_d}$. 

\begin{definition}[Dilation]
\label{dilation}
Let $b$ be a structuring element of $\mathcal{C}$. The {\bf dilation over $b$}, denoted by $\delta[b]$, is the mapping from $\mathcal{C}$ to $\mathcal{C}$ defined for every object $d \in |\mathcal{C}|$ as follows. Let $D^\delta_d$ be the diagram $\{e \in |\mathcal{C}| \mid \forall v \in U(d), b(v_t) \rightarrowtail e\}$, then  

$$\delta[b](d) = \bigwedge D^\delta_d$$
\end{definition}

\begin{example}
In the category of sets $Set$, given a set $S$, for every $A \subseteq S$, the dilation of $A$ w.r.t. $b$ is $\delta[b](A) = \bigcup_{x \in A} b(x)$. Hence, in the case where $S$ is an Abelian group with an additive law $+$, $\delta[b](A) = \bigcup_{x \in A} \{x + p \mid p \in B\}$ where $B$ is a structuring element (and $b(x) = B_x = x+B$). \\ Here again, this shows that Definition~\ref{dilation} includes the classical definition of $\delta$ by $B$ on sets (see Section~\ref{preliminaries} for a reminder of mathematical morphology on sets).
\end{example}

\begin{example}
Let us take again the example of the topos of directed graphs with the forgetful functor $U$ defined on vertices and the mapping $b$ as in Example~\ref{structuring element in toposes}. Let $G : \mathcal{B}^{op} \to Set$ be a graph. Let $[\alpha]$ be a subobject of $Sub(G)$. In this case, $\delta[b]([\alpha]) = [\beta]$ where $\beta : G' \Rightarrow G$ is the natural transformation with $G' : \mathcal{B}^{op} \to Set$ the graph defined by:

\begin{itemize}
\item $G'(V) = \{x \in G(V) \mid \forall e \in G(E), S_e = \{x,y\} \Rightarrow x \in \alpha_V(dom(\alpha)(V))~\mbox{or}~y \in \alpha_V(dom(\alpha)(V))\}$, and
\item $G'(E) = \{e \in G(E) \mid G(s)(e),G(t)(e) \in G'(V)\}$.
\end{itemize}
Hence, $\beta_V$ and $\beta_E$ are inclusion mappings.
\end{example}

\begin{proposition}
\label{prop:stability2}
We have both that $t \in D^\delta_d$ (and then $\bigvee D^\delta_d = t$) and $\delta[b](d) \in D^\delta_d$.
\end{proposition}

\begin{proof}
The first point is a direct consequence of the terminality of $t$. \\ In the same way, by definition of $D^\delta_d$, we have for every $v \in U(d)$ that $b(v_t) \rightarrowtail \delta[b](d)$, whence we can conclude that $\delta[b](d) \in D^\delta_d$. 
\end{proof}

Hence, $\delta[b](d)$ is the limit of the cocone $(e \rightarrowtail t)_{e \in D^\delta_d}$,  and reciprocally $t$ is the colimit of the cone $(\delta[b](d) \rightarrowtail e)_{e \in D^\delta_d}$. 

\begin{proposition}
\label{prop:stability3}
The dilation of any element $d$ of $|\mathcal{C}|$ can be decomposed as follows:
$\delta[b](d) = \bigvee \{b(v_t) \mid v \in U(d)\}$.
\end{proposition}

\begin{proof}
Obviously, we have that $\bigvee \{b(v_t) \mid v \in U(d)\} \in D^\delta_d$, and then by the properties of limits, $\delta[b](d) \rightarrowtail \bigvee \{b(v_t) \mid v \in U(d)\}$. \\ Now, by Proposition~\ref{prop:stability2}, we have that $\delta[b](d) \in D^\delta_d$, and then for every $v \in U(d)$, $b(v_t) \rightarrowtail \delta[b](d)$. Hence, by the properties of colimits, we can conclude that $\bigvee \{b(v_t) \mid v \in U(d)\} \rightarrowtail \delta[b](d)$.
\end{proof}

Proposition~\ref{prop:stability3} then defines an algorithm to compute dilation by iterating on the elements of $U(d)$.\footnote{Of course this requires that the set $U(d)$ is of finite cardinality, and the supremum of two elements is computable (what is the case for all the examples presented in the paper).} 

We will give in the next section the definition of an algorithm to compute erosions. However, to be able to compute erosions, some conditions on morpho-categories will be required.

\subsection{Results}
\label{sect:results}

\begin{theorem}[Adjunction]
\label{adjonction}
For all $d, e \in |\mathcal{C}|$, the following equivalence holds:
$$d \rightarrowtail \varepsilon[b](e) \Longleftrightarrow \delta[b](d) \rightarrowtail e$$
\end{theorem}

\begin{proof}
($\Longrightarrow$) Let us suppose that $d \rightarrowtail \varepsilon[b](e)$. By Proposition~\ref{prop:stability1}, we know that $\varepsilon[b](e) \in D^\varepsilon_e$. Moreover, $U$ preserves monics. We then have that $U(d \rightarrowtail \varepsilon[b](e))$ is an injective mapping. Hence, for every $v \in U(d)$, we have that $b(v_t) \rightarrowtail e$, and then we can deduce that $e \in D^\delta_d$. By the properties of limits, we then have that   $\delta[b](d) \rightarrowtail e$.

\medskip
($\Longleftarrow$) Let us suppose that $\delta[b](d) \rightarrowtail e$. As $U$ preserves monics, we have that $U(\delta[b](d) \rightarrowtail e)$ is an injective mapping. By Proposition~\ref{prop:stability1}, we further know that $\delta[b](d) \in D^\delta_d$. Hence, we have for every $v \in U(d)$ that $b(v_t) \rightarrowtail e$, whence we can conclude that $d \in D^\varepsilon_e$, and then by the properties of colimits we have that $d \rightarrowtail \varepsilon[b](e)$. 
\end{proof}

As the adjunction property holds between $\varepsilon[b]$ and $\delta[b]$, the following standard results in mathematical morphology also hold.

\begin{theorem}
\label{th:main results}
$\varepsilon[b]$ and $\delta[b]$ are:
\begin{itemize}
\item {\em monotonic:} for every $d,d' \in |\mathcal{C}|$, if $d \rightarrowtail d'$ then $\varepsilon[b](d) \rightarrowtail \varepsilon[b](d')$ and $\delta[b](d) \rightarrowtail \delta[b](d')$;
\item {\em commutative} with respect to $\bigwedge$ and $\bigvee$, respectively: for every diagram $\mathcal{D}$ in $\mathcal{C}$, $\varepsilon[b](\bigwedge \mathcal{D}) = \bigwedge \{\varepsilon[b](d) \mid d \in \mathcal{D}\}$ and $\delta[b](\bigvee  \mathcal{D}) = \bigvee \{\delta[b](d) \mid d \in \mathcal{D}\}$;
\item and satisfy the {\em preservation} properties: $\varepsilon[b](t) = t$ and $\delta[b](\emptyset) = \emptyset$;
\end{itemize}
\end{theorem}

\begin{proof}
We prove these properties for $\varepsilon[b]$. The proof for $\delta[b]$ is substantially similar. Note that while the sketch of the proof is similar to the classical proof in the general algebraic setting of mathematical morphology, the detailed proof is here adapted to the particular setting proposed in this paper and to its specificities.
\begin{itemize}
    \item {\em Monotonicity.} Obviously, we have that $\varepsilon[b](d) = \bigvee \{e \mid e \rightarrowtail \varepsilon[b](d)\}$. By adjunction, we then have that $\varepsilon[b](d) = \bigvee \{e \mid \delta[b](e) \rightarrowtail d\}$. If $d \rightarrowtail d'$, we then have that $\{e \mid \delta[b](e) \rightarrowtail d\} \subseteq \{e \mid \delta[b](e) \rightarrowtail d'\}$, and then $\bigvee \{e \mid \delta[b](e) \rightarrowtail d\} \rightarrowtail \bigvee \{e \mid \delta[b](e) \rightarrowtail d'\}$, whence we can conclude that $\varepsilon[b](d) \rightarrowtail \varepsilon[b](d')$. 
    \item {\em Commutativity.} Let $\mathcal{D}$ be a diagram of objects in $\mathcal{C}$. By the properties of limits, $\bigwedge \mathcal{D} \rightarrowtail d$ for every $d \in \mathcal{D}$. Therefore, by monotonicity, we have that $\varepsilon[b](\bigwedge \mathcal{D}) \rightarrowtail \varepsilon[b](d)$, and then by the property of limits, we can conclude that $\varepsilon[b](\bigwedge \mathcal{D}) \rightarrowtail \bigwedge \{\varepsilon[b](d) \mid d \in \mathcal{D}\}$. \\ Let us prove the opposite direction. First, let us observe that for every $v \in U(\bigwedge \{\varepsilon[b](d) \mid d \in \mathcal{D}\})$, we have that $b(v_t) \rightarrowtail d$ for every $d \in \mathcal{D}$. By the property of limits, we then have for every $v \in  U(\bigwedge \{\varepsilon[b](d) \mid d \in \mathcal{D}\})$, that $b(v_t) \rightarrowtail \bigwedge \mathcal{D}$. Hence, $\bigwedge \{\varepsilon[b](d) \mid d \in \mathcal{D}\}$ belongs to the diagram $D^\varepsilon_{\bigwedge \mathcal{D}}$ whence we conclude that $\bigwedge \{\varepsilon[b](d) \mid d \in \mathcal{D}\} \rightarrowtail \varepsilon[b](\bigwedge \mathcal{D})$.
    \item {\em Preservation.} Obviously, we have that $\varepsilon[b](c) \rightarrowtail c$. In the same way, we obviously have that $c \in D^\varepsilon_c$, and then by the properties of colimits, we can conclude that $c \rightarrowtail \varepsilon[b](c)$.
\end{itemize}
\end{proof}

By the monotonicity property, both $\varepsilon[b]$ and $\delta[b]$ are functors from $\mathcal{C}$ to itself, and then by Theorem~\ref{adjonction}, $\delta[b]$ is left-adjoint to $\varepsilon[b]$ in the category $\mathcal{C}$. 

\medskip
As usual, the composition of erosion and dilation is not equal to the identity, but produces two other operators, called opening (defined as $\delta[b] \circ \varepsilon[b]$) and closing (defined as $\varepsilon[b] \circ \delta[b]$). Opening and closing have the following properties.

\begin{theorem}
$\varepsilon[b] \circ \delta[b]$ (closing) and $\delta[b] \circ \varepsilon[b]$ (opening) satisfy the following properties:
\begin{itemize}
\item $\varepsilon[b] \circ \delta[b]$ is extensive;
\item $\delta[b] \circ \varepsilon[b]$ is anti-extensive;
\item $\varepsilon[b] \circ \delta[b] \circ \varepsilon[b] = \varepsilon[b]$;
\item $\delta[b] \circ \varepsilon[b] \circ \delta[b] = \delta[b]$;
\item $\varepsilon[b] \circ \delta[b]$ and $\delta[b] \circ \varepsilon[b]$ are idempotent. 
\end{itemize}
\end{theorem}

\begin{proof}
Here, the proof is the same as the classical one, and it is detailed here for the sake of completeness, using the notations of the categorial setting.
\begin{itemize}
\item {\em Extensivity of closing.} Obviously, we have that $\delta[b](d) \rightarrowtail \delta[b](d)$, and then by the adjunction property, we can conclude that $d \rightarrowtail \varepsilon[b](\delta[b](d))$.
\item {\em Anti-extensivity of opening.} Obviously, we have that $\varepsilon[b](d) \rightarrowtail \varepsilon[b](d)$, and then by the adjunction property, we can conclude that $\delta[b](\varepsilon[b](d)) \rightarrowtail d$.
\item {\em Preservation.} By anti-extensivity of opening we have that $\delta[b](\varepsilon[b](d)) \rightarrowtail d$, and then by monotonicity of erosion, we can conclude that

$$\varepsilon[b](\delta[b](\varepsilon[b](d))) \rightarrowtail \varepsilon[d](d)$$
To show the opposite direction, we obviously have that $\delta[b](\varepsilon[b](d)) \rightarrowtail \delta[b](\varepsilon[b](d))$, and then by the property of adjunction, we can conclude that $\varepsilon[b](d) \rightarrowtail \varepsilon[b](\delta[b](\varepsilon[b](d)))$. \\ The equation $\delta[b] \circ \varepsilon[b] \circ \delta[b] = \delta[b]$ can be proved similarly.
\item {\em Idempotence.} Direct consequence of the preservation properties.
%\Isa{il suffit de composer par $\varepsilon$ ou $\delta$ la propriete precedente. C'est direct.} \Marc{By extensivity of opening, we have that $d \rightarrowtail \varepsilon[d](\delta[b](d))$, and then by monotonicity of erosion and dilation, we have that $\varepsilon[d](\delta[b](d)) \rightarrowtail \varepsilon[d](\delta[b](\varepsilon[d](\delta[b](d))))$. \\ To show the opposite direction, we obviously have that $\varepsilon[b](\delta[b](d)) \rightarrowtail \varepsilon[b](\delta[b](d))$, and then by the adjunction property, we have that $\delta[b](\varepsilon[b](\delta[b](d))) \rightarrowtail \delta[b](d)$. By monotonicity of erosion, we can conclude that $\varepsilon[b](\delta[b](\varepsilon[b](\delta[b](d)))) \rightarrowtail \varepsilon[b](\delta[b](d))$. \\ The idempotence of $\delta[b] \circ \varepsilon[b]$ is proved similarly.}
\end{itemize}
\end{proof}

A property which can be sought by the operations of erosion and dilation is to be, respectively, anti-extensive (i.e. $\varepsilon[b](d) \rightarrowtail d$) and extensive (i.e. $d \rightarrowtail \delta[b](d)$). With no further condition on the structure of objects $c \in |\mathcal{C}|$, we cannot hope to have such results. Indeed, no structural link exists between the objects $e \in D^\varepsilon_d$ (resp. $d$) and the objects $b(v_t)$ for each $v \in U(e)$ (resp. $v \in U(d)$). It can be natural for an element $v \in U(c)$ to see the object $b(v)$ as a neighborhood of $v$, and which in our case can mean an object structured around the element $v$. This is similar to a cover notion.

\begin{definition}[Covered by]
Let $\mathcal{C}$ be a morpho-category. Let $b$ be a structuring element. Let $c \in |\mathcal{C}|$ be an object. We say that $c$ is {\bf covered by} $b$ if $c \rightarrowtail \bigvee \{b(v_t) \mid v \in U(c)\}$.  \\
We say that $\mathcal{C}$ is {\bf covered by} b if every object $c \in |\mathcal{C}|$ is covered by $b$. 
\end{definition}

\begin{theorem}
\label{extensivity}
Under the condition that $\mathcal{C}$ is covered by $b$, then $\varepsilon[b]$ and $\delta[b]$ are, respectively:

\begin{enumerate}
\item {\em anti-extensive:} for every $d \in |\mathcal{C}|$, $\varepsilon[b](d) \rightarrowtail d$;
\item {\em extensive:} for every $d \in \mathcal{C}$, $d \rightarrowtail \delta[b](d)$.
\end{enumerate}
\end{theorem}

\begin{proof}
\begin{itemize}
\item {\em Anti-extensivity.} By the hypothesis of cover for $\mathcal{C}$, we have that $e \rightarrowtail \bigvee \{b(v_t) \mid v \in U(e)\}$, and then $e \rightarrowtail d$ for every $e \in D^\varepsilon_d$. Therefore, by the property of colimits, we conclude that $\varepsilon[b](d) \rightarrowtail d$.
\item {\em Extensivity.} By the hypothesis of cover for $\mathcal{C}$ we have that $d \rightarrowtail \bigvee \{b(v_t) \mid v \in U(d)\}$, and then $d \rightarrowtail e$ for every $e \in D^\delta_d$. Therefore, by the property of limits, we conclude that $d \rightarrowtail \delta[b](d)$.
\end{itemize}
\end{proof}

%\begin{proposition}
%If $d \rightarrowtail d'$, then $\overline{d}' \rightarrowtail \overline{d}$.
%\end{proposition}

%\begin{proof}
%Let us suppose that $d \rightarrowtail d'$. Let us show that for every $d'_1 \in |\mathcal{C}|$ such that $\bigwedge \{d',d'_1\} = \emptyset$, we have that $d'_1 \in \{d_1 \mid \bigwedge \{d,d_1\} = \emptyset\}$. If we suppose that there exists $d'' \neq \emptyset$ such that $d'' = \bigwedge \{d,d'_1\}$, then we have that $d'' \rightarrowtail d'$ and $d'' \rightarrowtail d'_1$ which leads to a contradiction.  \\ Hence, $\overline{d}$ is an upper bound of $\{d'_1 \mid \bigwedge \{d',d'_1\} = \emptyset\}$, and then by the universal property of colimit we have that $\overline{d}' \rightarrowtail \overline{d}$.
%\end{proof}

Another property which can be sought is that erosion and dilation are dual operations. Of course, this requires first to be able to define the notion of complement what is made possible if  morpho-categories are Cartesian closed. Indeed, we have seen in Section~\ref{sec:morpho-categories} that given an object $d \in |\mathcal{C}|$, its complement $\overline{d}$ is the object $\emptyset^d$ (the exponential of the object $d$ and the initial object $\emptyset$ of $\mathcal{C}$).

\begin{theorem}
Let $\mathcal{C}$ be a morpho-category which is further Boolean and covered by $b$. Then, $\varepsilon[b](\overline{d}) = \overline{\delta[b](d)}$. 
\end{theorem}

\begin{proof}
By definition of erosion and dilation, we have:

\begin{enumerate}
\item $\varepsilon[b](\overline{d}) = \bigvee \{e \mid \forall v \in U(e), b(v_t) \rightarrowtail \overline{d}\}$
\item $\overline{\delta[b](d)} = \bigvee \{e \mid \exists v \in U(d), b(v_t) {\not \rightarrowtail} e\}$
\end{enumerate}
By the first point we have for every $e \in D^\varepsilon_{\overline{d}}$ that for every $v \in U(e)$,  $b(v_t) {\not \rightarrowtail} \overline{\overline{d}} = d$ ($\mathcal{C}$ is Boolean).  \\ As $\mathcal{C}$ is covered by $b$, this means that for every $e \in D^\varepsilon_{\overline{d}}$, there exists $v \in U(d)$ such that $b(v_t) {\not \rightarrowtail} e$, and then $D^\varepsilon_{\overline{d}} \subseteq \bigvee \{e \mid \exists v \in U(d), b(v_t) {\not \rightarrowtail} e\}$. 

\medskip
By following the same process, we can show that $\bigvee \{e \mid \exists v \in U(d), b(v_t) {\not \rightarrowtail} e\} \subseteq D^\varepsilon_{\overline{d}}$, and then  concluding that $\varepsilon[b](\overline{d}) = \overline{\delta[b](d)}$. 
\end{proof}

This cover condition added to atomicity allows us to further compute erosions. Indeed, we have the following property:

\begin{proposition}
\label{propo:algorithm erosion}
If $\mathcal{C}$ is covered by $b$ and has atoms for $U$, then 

$$\varepsilon[b](d) = 
\left\{
\begin{array}{c}
\bigvee \{c_{v_t} \mid v \in U(d)~\mbox{and}~b(v_t) \rightarrowtail d\} \\
\bigvee \\
\bigvee \{b(v_t) \mid v \in U(d)~\mbox{and}~b(v_t) \rightarrowtail d~\mbox{and}~\forall v' \in b(v_t), b(v'_t) \rightarrowtail d\}
\end{array}
\right.$$
\end{proposition}

\begin{proof}
As $\mathcal{C}$ is covered by $b$, we have both that $\bigvee \{c_{v_t} \mid v \in U(d)~\mbox{and}~b(v_t) \rightarrowtail d\}$ and $\bigvee \{b(v_t) \mid v \in U(d)~\mbox{and}~b(v_t) \rightarrowtail d~\mbox{and}~\forall v' \in b(v_t), b(v'_t) \rightarrowtail d\}$ belong to $D^\varepsilon_d$, and then so is their supremum, whence we can deduce that this supremum is smaller than $\varepsilon[b](d)$ with respect to $\rightarrowtail$. 
\\
Now, as $\mathcal{C}$ is covered by $b$, we have for every $e \in D^\varepsilon_d$ that for every $v \in U(e)$ both $b(v_t) \rightarrowtail d$ and for every $v' \in U(b(v_t))$ that $b(v'_t) \rightarrowtail d$. And then, we have that 

$$e \rightarrowtail 
\left\{
\begin{array}{c}
\bigvee \{c_{v_t} \mid v \in U(d)~\mbox{and}~b(v_t) \rightarrowtail d\} \\
\bigvee \\
\bigvee \{b(v_t) \mid v \in U(d)~\mbox{and}~b(v_t) \rightarrowtail d~\mbox{and}~\forall v' \in b(v_t), b(v'_t) \rightarrowtail d\}
\end{array}
\right.$$
whence we can conclude that 
$$\varepsilon[b](d) \rightarrowtail
\left\{
\begin{array}{c}
\bigvee \{c_{v_t} \mid v \in U(d)~\mbox{and}~b(v_t) \rightarrowtail d\} \\
\bigvee \\
\bigvee \{b(v_t) \mid v \in U(d)~\mbox{and}~b(v_t) \rightarrowtail d~\mbox{and}~\forall v' \in b(v_t), b(v'_t) \rightarrowtail d\}
\end{array}
\right.$$
\end{proof}

\begin{remark}
As $\mathcal{C}$ is supposed covered by $b$ in Proposition~\ref{propo:algorithm erosion}, we have the following property: $\forall v \in U(t), c_v \rightarrowtail b(v)$. And then, either $\varepsilon[b](d) = 
\bigvee \{c_{v_t} \mid v \in U(d)~\mbox{and}~b(v_t) \rightarrowtail d\}$ if for every $v \in U(d)$ such that $b(v_t) \rightarrowtail d$ there exists $v' \in U(b(v_t))$ such that $b(v'_t) {\not \rightarrowtail} d$, or  $\varepsilon[b](d) = \bigvee \{b(v_t) \mid v \in U(d)~\mbox{and}~b(v_t) \rightarrowtail d~\mbox{and}~\forall v' \in b(v_t), b(v'_t) \rightarrowtail d\}$ otherwise. 
\end{remark}

This gives rise to the following algorithm:

\medskip
\begin{algorithmic}[1]
\Procedure{Erosion}{$d \in |\mathcal{C}|$ and $b : U(t) \to |\mathcal{C}|$}
\State $S = \{v \in U(d) \mid b(v_t) \rightarrowtail d\}$
\State $c = \bigvee \{c_{v_t} \mid v \in S\}$ 
\While{$S$ is not empty}
    \State choose $v \in S$
    \If{$\forall v' \in U(b(v_t)), b(v'_t) \rightarrowtail d$}
    \State $c = c \vee b(v_t)$ 
    \EndIf
    \State $S = S \setminus \{v\}$
\EndWhile	
\State \Return $c$
\EndProcedure
\end{algorithmic}

\begin{corollary}
If $\mathcal{C}$ is covered by $b$ and has atoms for $U$, then the algorithm {\sc Erosion} returns $\varepsilon[b](d)$.
\end{corollary}

\begin{proof}
Direct consequence of Proposition~\ref{propo:algorithm erosion}.
\end{proof}

\section{Morpholizable categories}
\label{morpholizable categories}

Complete toposes allow us to generate a whole family of morpho-categories, the Heyting algebra $Sub(X)$ where $X$ is any object of the topos under consideration. The problem of these morpho-categories is that subobjects are equivalence classes of monics defined up to domain isomorphisms. To define our two mathematical morphology operations of erosion and dilation from a structuring element, the notions of point and neighborhood around this point are essential. We then need to choose one subobject among the set of isomorphic domains to get this set of points and to define the notion of structuring elements (see Example~\ref{Ex:morpho-categories}). But, in practice, chosen subobjects are linked by inclusion morphisms, and not by monics. 

Here, we will define a family of categories, the {\em morpholizable categories}, which as presheaf toposes will also allow us to generate a whole family of morpho-categories, but taking better account of the notion of inclusion between subobjects. 

\subsection{Definitions and results}

%\begin{notation}[Functor image]
%Given a functor $F : \mathcal{C} \to \mathcal{D}$, let us denote by $Im(F)$, called {\bf functor image}, the subcategory of $\mathcal{D}$ whose objects are $F(\mathcal{C})$ and morphisms are $Hom_{Im(F)}$ defined as the least subset of $Hom_\mathcal{D}$ containing $F(Hom_\mathcal{C})$ and closed under composition. 
%\end{notation}

%This composition closure constraint is imposed because, without it, we would not necessarily ensure that $Im (U_i)$ is a category. There are counter-examples~\footnote{Let $C$ be the category with four objects $a,b,c,d$ and two morphisms (other than identities) $a \to b$ and $c \to d$. Now, let $D$ be the category with three objetcs $x,y,z$ and two morphisms $x \to y$ and $y \to z$ with their composition $x \to z$. Let $F : C \to D$ be the functor which maps $a \mapsto x$, $b,c \mapsto y$, and $d \mapsto z$. $Im(F)$ is not a category because the morphism $x \to z$ has not antecedent in $C$.} often related to surjective functors on objects but not necessarily on morphisms.

\begin{definition}[Morpholizable category]
\label{morpholizable category}
A {\bf morpholizable category} $\mathcal{C}$ is a complete and co-complete category equipped with a family of faithful functors $(U_i)_{i \in I} : \mathcal{C} \to Set$ such that the following properties hold:

\begin{itemize}
\item {\em Preservation of bounds:} for every $i \in I$, $U_i$ preserves limits and colimits. 
\item {\em Preservation of equality:} for all objects $c,d \in |\mathcal{C}|$, if for every $i \in I$, $U_i(c) = U_i(d)$, then $c = d$.
\end{itemize}
\end{definition}

%If $\mathcal{C}$ is further a topos, then for every $i \in I$, the pair $(U_i \dashv \triangle_i)$ is a geometric morphism which is further local ($\nabla_i$ is right-adjoint to $U_i$). This brings us closer to the notion of cohesive topos~\cite{Law94,Law07} which is roughly a geometric space that is a set of points together with a structure called cohesion (e.g. topology, smooth structure). 

\begin{example}
\label{Ex:morpholizable}
Given a presheaf topos $Set^{\mathcal{B}^{op}}$, we can define the family of forgetful functors $(U_b : F \mapsto F(b))_{b \in |\mathcal{B}|}$. Given a diagram $\mathcal{D}$ in $Set^{\mathcal{B}^{op}}$, we can define the two presheaves $F_{lim},F_{colim} : \mathcal{B}^{op} \to Set$ which to every $b \in \mathcal{B}$ associates the limit and the colimit of the diagram $\mathcal{D}_b = \{F(b) \mid F \in \mathcal{D}\}$ (this makes sense because the category $Set$ is complete and cocomplete). Hence, each $U_b$ preserves both limits and colimits. Now, the family $(U_b : F \mapsto F(b))_{b \in |\mathcal{B}|}$ obviously preserves equality. \\ Conversely, we cannot assure that each $U_b$ is faithful for all presheaf toposes. The reason is given two natural transformations $\alpha,\beta : F \Rightarrow F'$, $\alpha \neq \beta$ only means that there exists $b \in \mathcal{B}$ such that $\alpha_b \neq \beta_b$, and then nothing prevents that there are $b' \neq b \in \mathcal{B}$ such that $\alpha_{b'} = \beta_{b'}$ (there is a priori no link between $F(b)$ and $F(b')$). This is of course not true for all the examples we have presented so far (sets, graphs, and hypergraphs) which proves that for theses examples each forgetful functor is faithful, and then all are morpholizable categories. \\
We will also develop another example of morpholizable category, the simplicial complexes, which can not be defined as a presheaf topos~\footnote{Simplicial complexes are particular hypergraphs, but unlike the category of hypergraphs, the category of simplicial complexes is not a topos but only a quasi-topos (i.e. it is locally cartesian closed and the subobject classifier $\Omega$ only classifies strong monics).}.  

\paragraph{\bf Simplicial Complexes.} Simplicial complexes are geometrical objects defined by combinatory data. They enable to specify topological spaces of arbitrary
dimensions. Simplicial complexes define particular hypergraphs~\footnote{Let us recall that a hypergraph is a set $V$ whose elements are called {\em vertices} together with a family $E = (E_i)_{i \in I}$ of subsets of $V$ called {\em hyperedges}. A morphism of hypergraphs between $(V,E)$ and $(W,F)$ is then a mapping $f : V \to W$ such that for every $E_i \in E$, $f(E_i) \in F$.}. Indeed, a simplicial complex $K$ consists of a set of vertices $V$ and a set $E$ of subsets of $V$ such that $E$ is closed under subsets, and for every $v \in V$, $\{v\} \in E$ and $\emptyset \notin E$. Let us note $\mathcal{K}$ the full subcategory of $\mathcal{H}$ whose the objects are simplicial complexes. It is known that the category of simplicial complexes has limits and colimits for small diagrams. The two forgetful functors  $U_V$ and $U_E$ are defined as for hypergraphs, and it is quite easy to show that they are faithful, and preserve limits and colimits, and equalities.
\end{example}

Since $\mathcal{C}$ has limits and colimits, it has terminal and initial objects which are respectively the limit and the colimit of the empty diagram. The initial object is denoted by $\emptyset$, and the terminal one is denoted by $\mathbb{1}$. 
 
 \medskip
 By standard results of category theory, the following properties are satisfied:
 
 \begin{proposition}
 \label{standard properties} 
 $\mathcal{C}$ satisfies the two following properties: for every $i \in I$
 \begin{enumerate}
 \item each $U_i$ preserves and reflects monics and epics;
 \item the pushout of monics is a monic. By duality, the pullback of monics is a monic. More generally, both limits and colimits of monics are monics.
 \end{enumerate}
 \end{proposition} 
 
 \begin{proof}
 \begin{enumerate}
  \item A consequence of the fact that $U_i$ is faithful is that each $U_i$ reflects monics and epics. This property, associated with the first property of Definition~\ref{morpholizable category}, leads to the fact that each $U_i$ preserves and reflects epics and monics.
 \item Let
 $$
\xymatrix{
 c \ar[r] \ar[d]_{monic} & d \ar[d]^{f} \\
 c' \ar[r] & d'
  }
$$
be a pushout diagram in $\mathcal{C}$. From the first property of Proposition~\ref{standard properties}, each $U_i$ preserves and reflects monics and epics. Therefore, since $U_i$ also preserves pushout (the first property of Definition~\ref{morpholizable category}), and since the pushout of monics is monic in $Set$, $U_i(f)$ is monic in $Set$. Since $U_i$ reflects monics, we have that $f$ is monic in $\mathcal{C}$.   The proofs for pullback, limits and colimits are similar. 
 \end{enumerate}
 \end{proof} 
 
 \begin{definition}[Inclusion]
 \label{inclusion}
Let $\mathcal{C}$ be a morpholizable category, with its family of faithful functors $(U_i)_{i \in I}$. A morphism $m : c \to c'$ in $\mathcal{C}$ is called an {\bf inclusion} when for every $i \in I$, each $U_i(m)$ is a set-theoretical inclusion. We will then denote such an inclusion by $c \hookrightarrow c'$. And we will denote by $\mathcal{I}$ the class of inclusions of $\mathcal{C}$. 
 \end{definition}
 
 \begin{proposition}
 \label{is a partial order}
 $\mathcal{I}$ is a partial order.
 \end{proposition}
 
 \begin{proof}
 %\Isa{dire que $c \hookrightarrow c$ est evident ? Quid de la transitivite (la composition des $U_i(m)$ reste une inclusion) ?}
 Reflexivity is obvious. Transitivity is a direct consequence of the facts that functors are structure-preserving maps between categories and that $U$ reflects monics.
 
For the anti-symmetry,
 it is sufficient to show that there exists at most one inclusion $c \hookrightarrow d$ between objects, and if there exists an inclusion $d \hookrightarrow c$, then $c = d$.
 
Then, let us suppose that there exist two inclusions $m : c \hookrightarrow d$ and $m' : c \hookrightarrow d$ in $\mathcal{C}$. Since each $U_i$ is faithful, we have that $U_i(m) \neq U_i(m')$ as $m \neq m'$ which is not possible for set inclusion.
 
Let us suppose that there are both $c \hookrightarrow d$ and $d \hookrightarrow c$. By definition of inclusion, we then have that $U_i(c) = U_i(d)$ for each $i \in I$. Moreover, as there is at most one inclusion between objects, we have that $c \simeq d$, and then by the preservation of equality condition, we can conclude that $c = d$. 
\end{proof}
 
Now, we are going to see how to extract for each object $c \in |\mathcal{C}|$ a category $\mathcal{C}^\mathcal{I}_c$ which will be a morpho-category (see Theorem~\ref{complete lattice}). 
 
\begin{proposition}
\label{is factorizable}
Every morphism $f : c \to d$ in $\mathcal{C}$ can be factorized  uniquely as $i \circ e$ where $i$ is an inclusion and $e$ is an epic.
\end{proposition}
 
\begin{proof}
Let $(d \rightrightarrows d \coprod_c d)$ where $d \coprod_c d$ is the pushout 
 
 $$
\xymatrix{
 c \ar[r]^f \ar[d]_f & d \ar[d] \\
 d \ar[r] & d \coprod_c d
  }
$$
Since $\mathcal{C}$ has limits, let us denote by $Im(f)$ the limit of $(d \rightrightarrows d \coprod_c d)$. By definition of limit, the limit $Im(f)$ comes with a morphism $m : Im(f) \to d$, and by the universality property of limits, we also have a unique  morphism $g : c \to Im(f)$. Since each $U_i$ preserves limits, $(U_i(Im(f)),U_i(m))$ is the limit of $(U_i(d) \rightrightarrows U_i(d) \coprod_{U_i(c)} U_i(d))$ in $Set$, and then $U_i(m) : U_i(Im(f)) \to U_i(d)$ is a monic and more precisely an inclusion, and $U_i(e) : U_i(c) \to U_i(Im(f))$ is a surjection. As each $U_i$ reflects both monics and epics, we can conclude that $m$ is an inclusion and $e$ an epic.
 \end{proof}
 
 Following~\cite{DGS91}, morpholizable categories are strongly inclusive. They then satisfy the following supplementary properties (see~\cite{DGS91} for their proofs):

 \begin{itemize}
 \item all isomorphisms are identities;
 \item any morphism which is both an inclusion and an epic is an identity;
 \item $\mathcal{I}$ is $\mathcal{I}$-right-cancellable in the sense that if $g \circ f \in \mathcal{I}$ and $g \in \mathcal{I}$, then $f \in \mathcal{I}$;
 \item $\mathcal{E}$ is $\mathcal{E}$-right-cancellable where $\mathcal{E}$ is the class of epics of $\mathcal{C}$ (i.e. $\mathcal{E}^{op}$ is $\mathcal{E}^{op}$-right-cancellable). 
 \end{itemize}
 
 We denote by $\mathcal{C}^\mathcal{I}$ the subcategory of $\mathcal{C}$ where $|\mathcal{C}^\mathcal{I}| = |\mathcal{C}|$ and morphisms are all the inclusions in $\mathcal{I}$. 
 
 \begin{proposition}
 \label{inclusions of limits and colimits}
 In $\mathcal{C}$, limits and colimits of inclusions are inclusions.
 \end{proposition}
 
 \begin{proof}
 In $Set$, it is known that limits and colimits of inclusions are inclusions. As each $U_i$ both preserves limits and colimits, and reflects monics, we can conclude that limits and colimits of inclusions in $\mathcal{C}$ are inclusions. 
 \end{proof}
 
 Since each $U_i$ reflects monics, $\mathcal{C}^\mathcal{I}$ has also $\emptyset$ as initial object. Conversely, $\mathcal{C}^\mathcal{I}$ has not necessarily a terminal object. The reason is that for most of morpholizable categories (anyway all the categories presented in this paper), the terminal object $\mathbb{1}$ has only one morphism to itself which is the identity $Id_{\mathbb{1}}$. This leads to the fact that $U_i(\mathbb{1})$ is a singleton, and then, as $U_i$ reflects epics, the unique morphism from any object $c$ to $\mathbb{1}$ is also an epic.  Hence, the category $\mathcal{C}^\mathcal{I}$ has limits but not necessarily colimits. By contrast, for any object $c \in |\mathcal{C}|$, the slice category $\mathcal{C}^\mathcal{I}/c$ has for initial object the inclusion $\emptyset \hookrightarrow c$ and for terminal object the identity $Id_c$. Moreover, by Proposition~\ref{inclusions of limits and colimits}, limits and colimits of inclusions are inclusions, and then we conclude that the slice category $\mathcal{C}^\mathcal{I}/c$ has limits and colimits. Obviously, there is the forgetful functor $U_c :  \mathcal{C}^\mathcal{I}/c \to \mathcal{C}$ which maps any inclusion $d \hookrightarrow c$ to $d$ and any inclusion $d \hookrightarrow d'$ 
 %\Isa{$d'$ est arbitraire ?} 
in $\mathcal{C}^\mathcal{I}/c$ to itself. This forgetful functor leads to a full subcategory $\mathcal{C}^\mathcal{I}_c$ of $\mathcal{C}^\mathcal{I}$. This subcategory is both complete and co-complete. Hence, given an object $c \in \mathcal{C}$, $\mathcal{C}^\mathcal{I}_c$ is the category whose objects are all the subobjects and the morphisms are all the inclusions between them. 
  
 \begin{theorem}
 \label{complete lattice}
 The category $\mathcal{C}^\mathcal{I}_c$ is a morpho-category.
 \end{theorem}
 
 \begin{proof}
The objects in $\mathcal{C}^\mathcal{I}_c$ are partially ordered by inclusions, and given two objects $d$ and $d'$, the supremum of $d$ and $d'$, denoted by $\bigvee \{d,d'\}$, is the coproduct of $d$ and $d'$ in $\mathcal{C}^\mathcal{I}_c$, and the infimum of $d$ and $d'$, denoted by $\bigwedge \{d,d'\}$, is the unique inclusion pullback of the inclusions $(d \hookrightarrow \bigvee \{d,d'\},d' \hookrightarrow \bigvee \{d,d'\})$. Hence, $\mathcal{C}^\mathcal{I}_c$ is a category finitely complete and finitely co-complete. Now, by Proposition~\ref{inclusions of limits and colimits}, this category is further complete and co-complete. Finally, by Proposition~\ref{standard properties}, each $U_i$ preserves monics and colimits.
 \end{proof}
 
 \begin{example}
In any presheaf category $Set^{\mathcal{B}^{op}}$, given a presheaf $F : \mathcal{B}^{op} \to Set$, the category $(Set^{\mathcal{B}^{op}})^\mathcal{I}_F$ has for objects all the presheaves $F' : \mathcal{B}^{op} \to Set$ such that for every $b \in \mathcal{B}$, $F'(b) \subseteq F(b)$. 
\end{example}

\subsection{Mathematical morphology over morpholizable categories}

Let $\mathcal{C}$ be a morpholizable category whose the family of faithful functors is $(U_i)_{i \in I}$. Let $c \in |\mathcal{C}|$ be an object. We have seen in the previous section that the category $\mathcal{C}^\mathcal{I}_c$ is a morpho-category. We can then define on it both operations of erosion and dilation based on structuring elements defined in Section~\ref{structuring}. In this framework, structuring elements are mappings $b : U_i(c) \to |\mathcal{C}^\mathcal{I}_c|$ for a given faithful functor $U_i$.

\medskip
Here, to facilite the presentation, we use the standard definition of sets, graphs and hypergraphs, rather than the one we followed in the previous section to define them by presheaves. 

 \begin{example}
\label{structuring element graph}
In the category of undirected graphs $\mathcal{G}$, if we consider the forgetful functor $U_V$, then given a graph $G = (V,E)$ the mapping $b$ can be defined for example as: for every vertex $x \in V$, $b(x) = G_x$ where $G_x = (V_x,E_x)$ is the subgraph of $G$ such that $V_x = \{y \mid \{x,y\} \in E\} \cup \{x\}$ and $E_x = \{\{x,y\} \in E \mid y \in V_x\}$. \\ Now, if we consider the forgetful functor $U_E$, then given a graph $G = (V,E)$, a possible mapping $b$ can be: for every edge $\{x,y\} \in E$, $b(\{x,y\}) = G_{\{x,y\}}$ where $G_{\{x,y\}} = (V_{\{x,y\}},E_{\{x,y\}})$ is the subgraph of $G$ such that $V_{\{x,y\}} = \{z \mid \{x,z\} \in E \; or \; \{y,z\} \in E\}$  and $E_{\{x,y\}} = \{\{x,z\} \in E \mid z \in V_{\{x,y\}}\} \cup \{\{y,z\}  \in E \mid z \in V_{\{x,y\}}\}$. 

%\Isa{on n'ajoute pas $x, y$ pour avoir $\delta$ extensive ? }
\end{example}

\begin{example}
\label{structuring element hypergraph}
In the category of hypergraphs $\mathcal{H}$, if we consider the forgetful functor $U_V$, then given a hypergraph $H = (V,E = (E_i)_{i \in I})$, the mapping $b$ can be defined as: for every vertex $x \in V$, $b(x) = H_x$ with $H_x = (V_x,E_x)$ is the sub-hypergraph of $H$ where $V_x = \bigcup \{E_i \mid i \in I, x \in E_i\}$ and $E_x = (E_j)_{j \in J}$ such that $J = \{i \in I \mid x \in E_i\}$. \\ 
If the forgetful functor is $U_E$, then given a hypergraph $H = (V,E = (E_i)_{i \in I})$, a possible $b$ can be: for every hyperedge $E_i \in E$, $b(E_i) = H_{E_i}$ where $H_{E_i} = (V_{E_i},E_{E_i})$ is the sub-hypergraph of $H$ such that $V_{E_i} = \bigcup \{E_j \mid  E_i \cap E_j \neq \emptyset\}$ and $E_{E_i} = \{E_j \in E \mid E_j \subseteq V_{E_i}\} = \{E_j \in E \mid E_i \cap E_j \neq \emptyset\}$.  
\end{example}

\begin{example}
\label{structuring element simplicial complex}
In the category of simplicial complexes $\mathcal{K}$, given a simplicial complex $K = (V,(E_i)_{i \in I})$, we can define for every $x \in V$, $b(x) = K_x$ with $K_x = (V_x,E_x)$ the sub-simplicial complex  of $K$ such that $V_x = \bigcup \{E_i \mid i \in I, x \in E_i\}$ and $E_x = \{E_i \neq \emptyset \mid \exists j \in I, x \in E_j \; and \;  E_i \subseteq E_j\}$. 
\end{example}

Given a structuring element $b : U_i(c) \to |\mathcal{C}^\mathcal{I}_c|$, the erosion over $b$ is then the mapping $\varepsilon[b]$ from $\mathcal{C}^\mathcal{I}_c$ to  $\mathcal{C}^\mathcal{I}_c$ which for every object $d \in |\mathcal{C}^\mathcal{I}_c|$ yields $\varepsilon[b](d) = \bigvee D^\varepsilon_d$ where $D^\varepsilon_d = \{e \in |\mathcal{C}^\mathcal{I}_c| \mid \forall v \in U_i(e), b(v) \hookrightarrow d\}$. 

\begin{example}
In the category of graphs $\mathcal{G}$, if we forget on the set of vertices, given a graph $G = (V,E)$ and considering the structuring element defined in Example~\ref{structuring element graph}, for every subgraph $G' = (V',E')$ the erosion of $G'$ is $\varepsilon[b](G') = (\overline{V},\overline{E})$ where
$$\overline{V} = \{x \in V \mid G_x \hookrightarrow G'\}$$ 
$$\overline{E} = \{\{x,y\} \mid x,y \in \overline{V}\}$$ 

When we forget on edges, we have that $\varepsilon[b](G') = (\overline{V},\overline{E})$ where 
$$\overline{V} = \{x \in V \mid \exists y \in V, \{x,y\} \in E' \; and \; \forall y \in V, \{x,y\} \in E' \Rightarrow G_{\{x,y\}} \hookrightarrow G'\}$$
$$\overline{E} = \{\{x,y\}\in E'  \mid x,y \in \overline{V}\}$$ 
\end{example}

\begin{example}
In the category of hypergraphs $\mathcal{H}$, if we forget on vertices, given a hypergraph $H = (V,E = (E_i)_{i \in I})$ and considering the structuring element defined in Example~\ref{structuring element hypergraph}, for every sub-hypergraph $H' = (V',E')$, the erosion of $H'$ is $\varepsilon[b](H') = (\overline{V},\overline{E})$ where $$\overline{V} = \{x \in V \mid H_x \hookrightarrow H'\}$$ $$\overline{E} = \{E_i \mid i \in I \; and \;  E_i \subseteq \overline{V}\}$$

When we forget on hyperedges, we have that $\varepsilon[b](H') = (\overline{V},\overline{E})$ where 
$$\overline{V} = \{x \in V \mid \exists i \in I', x \in E'_i \; and \; \forall j \in I', x \in E'_j \Rightarrow H_{E'_j} \hookrightarrow H'\}$$ 
$$\overline{E} = \{E'_i \in E' \mid E'_i \subseteq \overline{V}\}$$
\end{example}

\begin{example}
In the category of simplicial complexes $\mathcal{K}$, given a simplicial complex $K = (V,E = (E_i)_{i \in I})$, for $b$ defined as in Example~\ref{structuring element simplicial complex},  for every sub-simplicial complex $K' = (V',E')$, the erosion $\varepsilon[b](K')$ of $K'$ is  defined as for hypergraphs when we forget on vertices. We find the definition of the operator $Cl^A$ given in~\cite{DCN11} where given a sub-simplicial complex $X$, $Cl^A(X)$ is the largest simplicial complex contained in $X$. 
\end{example}

In the same way, given a structuring element $b : U_i(c) \to |\mathcal{C}^\mathcal{I}_c|$, the dilation over $b$ is the mapping $\delta[b]$ from $\mathcal{C}^\mathcal{I}_c$ to  $\mathcal{C}^\mathcal{I}_c$ which to every object $d \in |\mathcal{C}^\mathcal{I}_c|$ yields $\delta[b](d) = \bigwedge D^\delta_d$ where $D^\delta_d = \{e \in |\mathcal{C}^\mathcal{I}_c| \mid \forall v \in U(d), b(v) \hookrightarrow e\}$. 

\begin{example}
In the category of graphs $\mathcal{G}$, if we forget on the set of vertices, given a graph $G = (V,E)$ and considering the structuring element defined in Example~\ref{structuring element graph}, for every subgraph $G' = (V',E')$ the dilation of $G'$ is $\delta[b](G') = (\overline{V},\overline{E})$ where 
$$\overline{V} = \{x \in V \mid V_x \cap V' \neq \emptyset\}$$
$$\overline{E} = \{\{x,y\} \in E \mid x,y \in \overline{V}\}$$

When we forget on edges, we have that $\delta[b](G') = (\overline{V},\overline{E})$ where 
$$\overline{V} = \{z \in V \mid \exists e \in E, S_e \cap V' \neq \emptyset \; and \; z \in V_e\}$$
%\textcolor{red}{je proposerais de l'ecrire sous la forme suivante (pour que cela ressemble plus a ce qui precede) :}
%$$\overline{V} = \{z \in V \mid \exists \{x,y\} \in E, \{x,y\} \cap V' \neq \emptyset \; and \; z \in V_{\{x,y\}}\}$$
$$\overline{E} = \{\{x,y\} \in E \mid x,y \in \overline{V}\}$$ 
Note that this dilation is not extensive: if $G'$ contains isolated vertices, these vertices do not belong to the dilated graph, which makes sense with this definition of $b$ and associated dilation, purely based on edges.

A simple example illustrates these definitions in Figure~\ref{fig:graph}.

\begin{figure}[htbp]
\centerline{\includegraphics[width=5cm]{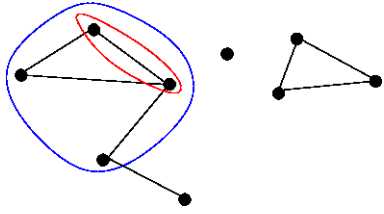}}
\caption{A simple example of a graph dilation. The graph $G$ is composed of a set of vertices $V$ represented by black dots, and a set of edges $E$ represented by black lines. The subgraph $G'=(V',E')$ has two vertices and one edge joining them (inside the red line). Its dilation $\delta[b](G')$ has 4 vertices and 4 edges, enclosed in the blue line. Now if $V'$ contains also the isolated vertex, then this vertex would belong to the dilation if we forget on vertices, but not if we forget on edges, illustrating the difference between both examples.}
\label{fig:graph}
\end{figure}

\end{example}

\begin{example}
\label{Ex:hypergraphs}
In the category of hypergraphs $\mathcal{H}$, if we forget on vertices, given a hypergraph $H = (V,E = (E_i)_{i \in I})$ and considering the structuring element defined in Example~\ref{structuring element hypergraph}, for every sub-hypergraph $H' = (V',E')$, the dilation of $H'$ is $\delta[b](H') = (\overline{V},\overline{E})$ where 
$$\overline{V} = \{x \in V \mid V_x \cap V' \neq \emptyset\}$$
$$\overline{E} = \{E_i  \in E \mid E_i \subseteq \overline{V}\}$$  This example corresponds to one of the dilations proposed in~\cite{BB13}. A simple illustration is given in Figure~\ref{fig:exCVIU}. 

When we forget on hyperedges, we have that $\delta[b](H') = (\overline{V},\overline{E})$ where 
$$\overline{V} = \{z \in V \mid \exists E_i \in E, E_i \cap V' \neq \emptyset \; and \; z \in E_i\}$$ $$\overline{E} = \{E_i \in E \mid E_i \subseteq \overline{V}\}$$ 

\begin{figure}[htbp]
\centerline{\includegraphics[width=10cm]{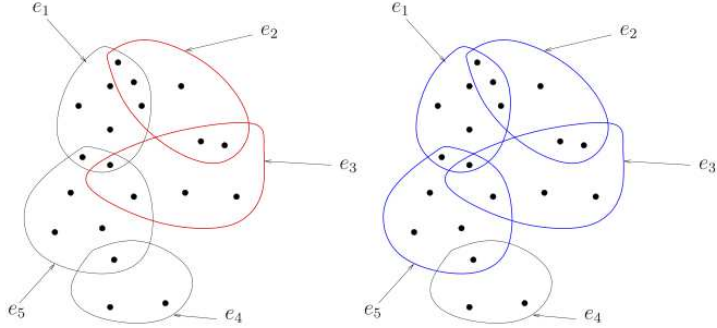}}
\caption{Illustration of the dilation of an hypergraph by forgetting on vertices. The figure on the left represents $V$ (vertices represented as points) and $E$ (hyperedges represented as closed lines). The red lines indicate the hyperedges of $H'$. The vertices of $H'$ are the points enclosed in these lines. The blue lines on the right represent the hyperedges of $\delta[b](H')$ and its vertices are the points enclosed in these lines. Figure from~\cite{BB13}.}
\label{fig:exCVIU}
\end{figure}

%\Isa{est-ce que cet exemple suffit ou il faut montrer des liens pour tout (graphes, hypergraphes, erosions, dilations en oubliant sur les sommets ou sur les aretes) ?}
\end{example}

%\Isa{
%\begin{example}
%idem for simplicial complexes ?
%\end{example}
%}

Since simplicial complexes define a subclass of hypergraphs, similar definitions as in Example~\ref{Ex:hypergraphs} for erosion and dilation can be given for this subclass.
\section{Application: morpho-logic}
\label{morpho-logic}

In~\cite{Bloch02}, it was shown that standard modalities $\Box$ and $\Diamond$ of the propositional modal logic can be defined as morphological erosion and dilation based on structuring elements defined over the  lattice $\mathcal{P}(\Omega)$ where $\Omega$ is a set of possible worlds. Here, we propose to extend this formalization to any morpho-category. To allow for reasoning, we must impose the constraints on morpho-categories to be Cartesian closed and covered by the structuring element.

\paragraph{Syntax.} Let $P$ be a countable set whose elements are called {\bf propositional variables} and denoted by letters $p,q,r,\ldots$ The set $\mathcal{F}$ of formulas is defined by the following grammar:
$$\varphi,\psi ::= \top | \bot | p | \neg \varphi | \varphi \wedge \psi | \varphi \vee \psi | \varphi \Rightarrow \psi | \Box \varphi | \Diamond \varphi$$
where $p$ runs through $P$. 

\paragraph{Semantics.} A {\bf model} $\mathcal{M}$ is a triple $(\mathcal{C},b,\nu)$ where:

\begin{itemize}
\item $\mathcal{C}$ is a morpho-category which is further Cartesian closed and covered by $b$,
\item $b : U(t) \to |\mathcal{C}|$ is a structuring element, and
\item $\nu : P \to |\mathcal{C}|$ is a mapping called {\bf valuation}.
\end{itemize}
The semantics of formulas in a model $\mathcal{M}$ is a mapping $\sem{\mathcal{M}}(\_) : \mathcal{F} \to |\mathcal{C}|$ defined by structural induction on formulas as follows:

\begin{itemize}
\item $\sem{\mathcal{M}}(\top) = t$ ($t$ is the terminal object of $\mathcal{C}$),
\item $\sem{\mathcal{M}}(\bot) = \emptyset$ ($\emptyset$ is the initial object of $\mathcal{C}$),
\item $\sem{\mathcal{M}}(p) = \nu(p)$,
\item $\sem{\mathcal{M}}(\neg \varphi) = \emptyset^{\sem{\mathcal{M}}(\varphi)}$ (i.e. the exponential of $\emptyset$ and $\sem{\mathcal{M}}(\varphi)$),
\item $\sem{\mathcal{M}}(\varphi \wedge \psi) = \sem{\mathcal{M}}(\varphi) \wedge \sem{\mathcal{M}}(\psi)$ (the infimum of $\sem{\mathcal{M}}(\varphi)$ and $\sem{\mathcal{M}}(\psi)$), 
\item $\sem{\mathcal{M}}(\varphi \vee \psi) = \sem{\mathcal{M}}(\varphi) \vee \sem{\mathcal{M}}(\psi)$ (the supremum of $\sem{\mathcal{M}}(\varphi)$ and $\sem{\mathcal{M}}(\psi)$),
\item $\sem{\mathcal{M}}(\varphi \Rightarrow \psi) = \sem{\mathcal{M}}(\psi)^{\sem{\mathcal{M}}(\varphi)}$ (the exponential of $\sem{\mathcal{M}}(\psi)$ and $\sem{\mathcal{M}}(\varphi)$),
\item $\sem{\mathcal{M}}(\Box \varphi) = \varepsilon[b](\sem{\mathcal{M}}(\varphi)$ (the erosion of $\sem{\mathcal{M}}(\varphi)$ with respect to $b$),
\item $\sem{\mathcal{M}}(\Diamond \varphi) = \delta[b](\sem{\mathcal{M}}(\varphi)$ (the dilation of $\sem{\mathcal{M}}(\varphi)$ with respect to $b$)
\end{itemize}
We write $\mathcal{M} \models \varphi$ if and only if $\sem{\mathcal{M}}(\varphi) \simeq t$. And then, standardly, given a set of formulas $\Gamma$ and a formula $\varphi$, we write $\Gamma \models \varphi$ to mean that for every model $\mathcal{M}$ which for every formula $\psi \in \Gamma$ verifies $\mathcal{M} \models \psi$, we have that $\mathcal{M} \models \varphi$.

\medskip
Here, a standard Kripke model $(Q,R,\nu)$ where $Q$ is a set of states, $R \subseteq Q \times Q$ is an accessibility relation, and $\nu : P \to \mathcal{P}(Q)$ is a mapping, is represented  by the model $\mathcal{M} = (\mathcal{P}(Q),b,\nu)$ where $b : Q \to \mathcal{P}(Q)$ is defined by: $\forall q \in Q, b(q) = \{q' \in Q \mid (q,q') \in R\}$.

\medskip
As is customary, our semantics brings back the implication to the order relation $\rightarrowtail$ in $\mathcal{C}$.

\begin{lemma}
\label{implication lemma}
$\mathcal{M} \models \varphi \Rightarrow \psi$ iff $\sem{\mathcal{M}}(\varphi) \rightarrowtail \sem{\mathcal{M}}(\psi)$.
\end{lemma}

\begin{proof}
Cartesian closed morpho-categories internally define Heyting algebras, and then they satisfy the following property: $$\sem{\mathcal{M}}(\varphi) \rightarrowtail \sem{\mathcal{M}}(\psi) \Longleftrightarrow \sem{\mathcal{M}}(\psi)^{\sem{\mathcal{M}}(\varphi)} \simeq t$$
\end{proof}

\medskip
By the constraint for $\mathcal{C}$ to be Cartesian closed, the logic defined here is intuitionistic. Hence, the propositional connectives $\wedge$ and $\vee$ and the modalities $\Box$ and $\Diamond$ are not dual. 

\paragraph{A sound and complete entailment system.} Here, we propose to establish the consequence relation $\vdash$, called {\bf proofs}, between set of formulas and formulas. We then consider the following Hilbert-system.

\begin{itemize}
\item {\bf Axioms:}
\begin{itemize}
\item {\em Tautologies:} all intuitionistic propositional tautology instances~\footnote{We call intuitionistic propositional tautology instance any formula $\varphi$ such that there exists an intuitionistic propositional tautology $\psi$ whose propositional variables are among $\{p_1,\ldots,p_n\}$ and $n$ formulas $\varphi_1,\ldots,\varphi_n$ such that $\varphi$ is obtained by replacing in $\psi$ all the occurrences of $p_i$ by $\varphi_i$ for $i \in \{1,\ldots,n\}$. It is not difficult to show that if $\psi$ is a propositional tautology, then for every model $\mathcal{M}$, $\sem{\mathcal{M}}(\varphi) \simeq t$.},
\item {Preservation:}
\begin{itemize}
\item $\Diamond \bot \Leftrightarrow \bot$
\item $\Box \top \Leftrightarrow \top$
\end{itemize}
\item {\em Commutativity:}
	\begin{itemize}
	\item $\Box(\varphi \wedge \psi) \Leftrightarrow \Box \varphi \wedge \Box \psi$
	\item $\Diamond(\varphi \vee \psi) \Leftrightarrow \Diamond \varphi \vee \Diamond \psi$
	\end{itemize}	
\item {\em Anti-extensivity:} 
\begin{itemize}
\item $\Box \varphi \Rightarrow \varphi$
\item $\Diamond \Box \varphi \Rightarrow \varphi$
\end{itemize}
\item {\em Extensivity:} 
\begin{itemize}
\item $\varphi \Rightarrow \Diamond(\varphi)$
\item $\varphi \Rightarrow \Box \Diamond \varphi$
\end{itemize}
%\item {\em adjunction:} $\varphi \Rightarrow \Box \psi \Longleftrightarrow \Diamond \varphi \Rightarrow \psi$
\end{itemize}
\item {\bf Inference rules:}
\begin{itemize}
\item {\em Modus Ponens:} 
$$\frac{\varphi \Rightarrow \psi ~~~\varphi}{\psi}$$
\item {\em Necessity:} 
$$\frac{\varphi}{\Box \varphi}$$
\end{itemize}
\end{itemize}

%The inference rules and axioms given above define the system $IT$. The system $IS4$, $IB$ and $IS5$ ($I$ for intuitionistic) can be obtained by adding respectively the axioms written in our framework as follows:

%\begin{itemize}
%   \item $\Box \varphi \Rightarrow \Box \Box \varphi$ (IS4)
%    \item $\varphi \Rightarrow \Box \Diamond \varphi$ (IB)
%    \item $\Diamond \varphi \Rightarrow \Box \Diamond \varphi$ (IS5)
%\end{itemize}

Derivation is defined as usual, that is, given a formula $\varphi \in \mathcal{F}$, $\varphi$ is {\em derivable}, written $\vdash \varphi$, if $\varphi$ is one of the axioms, or follows from derivable formulas through applications of the inference rules. 

The proof of completeness that we present here follows Henkin's method~\cite{Henkin1949}. This method relies on the proof that every consistent set of formulas has a model. This relies on the deduction theorem which is known to fail for modal logics except under some conditions~\cite{HN12}. Here, we restrict the notion of derivation of a formula $\varphi$ from a set of formulas $\Gamma$, written $\Gamma \vdash \varphi$, to the notion of {\em local derivation} defined by: $\Gamma \vdash \varphi$ iff there exists a finite subset $\{\varphi_1,\ldots,\varphi_n\} \subseteq \Gamma$ such that $\vdash \varphi_1 \wedge \ldots \wedge \varphi_n \Rightarrow \varphi$.   

\begin{theorem}
The proof system defined above is sound, i.e. if $\Gamma \vdash \varphi$, then $\Gamma \models \varphi$.
\end{theorem}

\begin{proof}
First, let us prove that $\vdash \varphi \Longrightarrow \models \varphi$. The soundness of axioms is obvious by Lemma~\ref{implication lemma} and the results proved in Section~\ref{sect:results}. 

Soundness of Modus Ponens is obvious, and soundness of the Necessity rule and Monotony rule are direct consequences of Theorem~\ref{th:main results}.

To finish the proof, let us suppose that $\Gamma \vdash \varphi$. Let $\mathcal{M}$ be a model of $\Gamma$. By definition, $\Gamma \vdash \varphi$ means that there exists $\{\varphi_1,\ldots,\varphi_n\} \subseteq \Gamma$ such that $\varphi_1 \wedge \ldots \wedge \varphi_n \Rightarrow \varphi$. By the previous soundness result, we have that $\mathcal{M} \models \varphi_1 \wedge \ldots \wedge \varphi_n \Rightarrow \varphi$. As $\{\varphi_1,\ldots,\varphi_n\} \subseteq \Gamma$, we have that $\sem{\mathcal{M}}(\varphi_1 \wedge \ldots \wedge \varphi_n) \approx t$, and the by Lemma~\ref{implication lemma}, $\sem{\mathcal{M}}(\varphi) \approx t$ (i.e. $\mathcal{M} \models \varphi$). 
\end{proof}

Thanks to this definition of local derivability, we get the deduction theorem.

\begin{proposition}[Deduction theorem]
Let $\Gamma \subseteq \mathcal{F}$ be a set of assumptions. Then, we have $\Gamma \cup \{\varphi\} \vdash \psi$ if, and only if $\Gamma \vdash \varphi \Rightarrow \psi$.
\end{proposition}

\begin{proof}
The necessary condition is obvious, and can be easily obtained by Modus Ponens and monotonicity (if $\Gamma \vdash \varphi$ and $\Gamma \subseteq \Gamma'$, then $\Gamma' \vdash \varphi$ - see Proposition~\ref{prop:chellas} below). 

Let us prove the sufficient condition. Let us suppose that $\Gamma \cup \{\varphi\} \vdash \psi$. This means that there exists $\{\varphi_1,\ldots,\varphi_n\} \subseteq \Gamma$ such that $\vdash \varphi_1 \wedge \ldots \wedge \varphi_n \Rightarrow \psi$. Two cases have to be considered:
\begin{enumerate}
    \item there exists $i$, $1 \leq i \leq n$ such that $\varphi = \varphi_i$. The exportation rule $p \wedge q \Rightarrow r \Leftrightarrow p \Rightarrow q \Rightarrow r$ is a valid rule in propositional logic. Therefore, we can write that $\vdash \varphi_1 \ldots,\varphi_n \Rightarrow \varphi_i \Rightarrow \psi$, i.e. $\Gamma \vdash \varphi \Rightarrow \psi$. 
    \item For every $i$, $1 \leq i \leq n$, $\varphi \neq \varphi_i$. We can then write that $\vdash \varphi_1 \wedge \ldots \wedge \varphi_n \wedge \varphi \Rightarrow \psi$. By the exportation rule, we can also write $\vdash \varphi_1 \wedge \ldots \wedge \varphi_n \Rightarrow \varphi \Rightarrow \psi$, that is $\Gamma \vdash \varphi \Rightarrow \psi$. 
\end{enumerate}
\end{proof}

\begin{corollary}
For every $\Gamma \subseteq \mathcal{F}$ and every $\varphi \in \mathcal{F}$, $\Gamma \vdash \varphi$ if and only if $\Gamma \cup \{\neg \varphi\} \vdash \bot$.
\end{corollary}

\begin{proof}
The "$\Rightarrow$" part is obvious. Let us prove the "$\Leftarrow$" part. Let us suppose that $\Gamma \cup \{\neg \varphi\} \vdash \bot$. This means that there is a finite subset $\{\varphi_1,\ldots,\varphi_n\} \subseteq \Gamma \cup \{\neg \varphi\}$ such that $\varphi_1 \wedge \ldots \wedge \varphi_n \Rightarrow \bot$, and then by the propositional logic, we have that $\varphi_1 \wedge \ldots \wedge \varphi_n \Rightarrow \varphi$ (i.e. $\Gamma \cup \{\neg \varphi\} \vdash \varphi$). By the deduction theorem, we can then write $\Gamma \vdash \neg \varphi \Rightarrow \varphi$. The formula $(\neg \varphi \Rightarrow \varphi) \Rightarrow \varphi$ is a tautology, and then by Modus Ponens we have that $\Gamma \vdash \varphi$. 
\end{proof}

We also find the other standard properties associated to theoremhood, deducibility and consistency that we summarize in the following proposition.

\begin{proposition}
\label{prop:chellas}
\mbox{}

\begin{enumerate}
    \item $\vdash \varphi$ iff for every $\Gamma \subseteq \mathcal{F}$, $\Gamma \vdash \varphi$.
    \item if $\Gamma \vdash \Gamma'$ and $\Gamma' \vdash \varphi$, then $\Gamma \vdash \varphi$.
    \item if $\varphi \in \Gamma$, then $\Gamma \vdash \varphi$.
    \item if $\Gamma \vdash \varphi$ and $\Gamma \subseteq \Delta$, then $\Delta \vdash \varphi$.
    \item $\Gamma \vdash \varphi$ iff there exists a finite subset $\Delta$ of $\Gamma$ such that $\Delta \vdash \varphi$.
    \item if $\Gamma$ is consistent, then there exists a formula $\varphi \in \mathcal{F}$ such that not $\Gamma \vdash \varphi$.
    \item $\Gamma$ is consistent iff there is no $\varphi \in \mathcal{F}$ such that both $\Gamma \vdash \varphi$ and $\Gamma \vdash \neg \varphi$.
    \item $\Gamma \cup \{\varphi\}$ is consistent iff not $\Gamma \vdash \neg \varphi$.
\end{enumerate}
\end{proposition}

\begin{proof}
See the proof of Theorem 2.16 in \cite{Chellas80}.
\end{proof}

For example, from the inference system and all the properties associated with it, we can prove the following statements:

\begin{itemize}
    \item {\em Monotony:}
    \begin{itemize}
        \item $(\varphi \Rightarrow \psi) \Rightarrow (\Box \varphi \Rightarrow \Box \psi)$
        \item $(\varphi \Rightarrow \psi) \Rightarrow (\Diamond \varphi \Rightarrow \Diamond \psi)$
    \end{itemize}
    \item {\em Adjunction:} $(\varphi \Rightarrow \Box \psi) \Leftrightarrow (\Diamond \varphi \Rightarrow \psi)$
    \item {\em Kripke Schema:} $\Box(\varphi \Rightarrow \psi) \Rightarrow (\Box \varphi \Rightarrow \Box \psi)$
\end{itemize}

For instance the adjunction property is proved from the following two sequences:

\begin{itemize}
    \item $\varphi \Rightarrow \Box \psi \vdash \varphi \Rightarrow \Box \psi$
    \item $\varphi \Rightarrow \Box \psi \vdash (\varphi \Rightarrow \Box \psi) \Rightarrow (\Diamond \varphi \Rightarrow \Diamond \Box \psi)$ (Commutativity of $\Diamond$)
    \item $\varphi \Rightarrow \Box \psi \vdash \Diamond \varphi \Rightarrow \Diamond \Box \psi$ (Modus Ponens)
    \item $\varphi \Rightarrow \Box \psi \vdash \Diamond \Box \psi \Rightarrow \psi$ (anti-extensivity of $\Diamond \Box$)
    \item $\varphi \Rightarrow \Box \psi \vdash \Diamond \varphi \Rightarrow \psi$ (transitivity of implication)
    \item $\vdash (\varphi \Rightarrow \Box \psi) \Rightarrow (\Diamond \varphi \Rightarrow \psi)$ (Deduction theorem)
\end{itemize}

Similarly, we have 

\begin{itemize}
    \item $\Diamond \varphi \Rightarrow \psi \vdash \Diamond \varphi \Rightarrow \psi$
    \item $\Diamond \varphi \Rightarrow \psi \vdash (\Diamond \varphi \Rightarrow \psi) \Rightarrow (\Box \Diamond \varphi \Rightarrow \Box \psi)$ (Commutativity of $\Box$)
    \item $\Diamond \varphi \Rightarrow \psi \vdash \Box \Diamond \varphi \Rightarrow \Box \psi$ (Modus Ponens)
    \item $\Diamond \varphi \Rightarrow \psi \vdash \varphi \Rightarrow \Box \Diamond \varphi$ (extensivity of $\Box \Diamond$)
    \item $\Diamond \varphi \Rightarrow \psi \vdash \varphi \Rightarrow \Box \psi$ (transitivity of implication)
    \item $\vdash (\Diamond \varphi \Rightarrow \psi) \Rightarrow (\varphi \Rightarrow \Box \psi)$ (Deduction theorem)
\end{itemize}

\begin{definition}[Maximal Consistence]
A set of formulas $\Gamma \subseteq \mathcal{F}$ is {\bf maximally consistent} if it is consistent (i.e. we have not $\Gamma \vdash \bot$) and there is no consistent set of formulas properly containing $\Gamma$ (i.e. for each formula $\varphi \in \mathcal{F}$, either $\varphi \in \Gamma$ or $\neg \varphi \in \Gamma$, but not both).
\end{definition}

\begin{proposition}
\label{lindenbaum lemma}
Let $\Gamma \subseteq \mathcal{F}$ be a consistent set of formulas. There exists a maximally consistent set of formulas $\overline{\Gamma} \subseteq \mathcal{F}$  that contains $\Gamma$.
\end{proposition}

\begin{proof}
Let $S = \{\Gamma' \subseteq \mathcal{F} \mid \Gamma'~\mbox{is consistent and}~\Gamma \subseteq \Gamma'\}$. The poset $(S,\subseteq)$ is inductive. Therefore, by Zorn's lemma, $S$ has a maximal element $\overline{\Gamma}$. By definition of $S$, $\overline{\Gamma}$ is consistent and contains $\Gamma$. Moreover, it is maximal. Otherwise, there exists a formula $\varphi \in \mathcal{F}$ such that $\varphi \notin \overline{\Gamma}$. As $\overline{\Gamma}$ is maximal, this means that $\overline{\Gamma} \cup \{\varphi\}$ is inconsistent, and then $\overline{\Gamma} \cup \{\neg \varphi\}$ is consistent. As $\overline{\Gamma}$ is maximal, we can conclude that $\neg \varphi \in \overline{\Gamma}$.   
\end{proof}

To show completeness, we need first to build a canonical model whose behavior is equivalent to derivabilty.

\begin{definition}[Canonical model]
Let $\Gamma \subseteq \mathcal{F}$ be a set of formulas. $\mathcal{M}_\Gamma = (\mathcal{C},b,\nu)$ is the {\bf canonical model} over $\Gamma$ defined by:

\begin{itemize}
    \item the lattice $\mathcal{C}$ whose the elements are all maximally consistent sets $\overline{\Gamma}'$ with $\Gamma' \subseteq \Gamma$. The infimum (resp. supremum) of two maximally consistent sets $\overline{\Gamma'}$ and $\overline{\Gamma''}$ is given by $\overline{\Gamma' \cap \Gamma''}$ (resp. $\overline{\Gamma' \cup \Gamma''}$). It is quite obvious to show that this lattice is complete and it satisfies the infinite distributive law. Finally, given two maximally consistent sets $\Gamma'$ and $\Gamma''$ we define their exponential $\overline{\Gamma'}^{\overline{\Gamma''}}$ by $\bigvee \{\overline{\Gamma'''} \mid \overline{\Gamma'''} \wedge \overline{\Gamma''} \subseteq \overline{\Gamma'}\}$. As the infinite distributive law is satisfied, $\overline{\Gamma'}^{\overline{\Gamma''}}$ is the exponential of $\overline{\Gamma'}$ and $\overline{\Gamma''}$. 
    \item $b : \overline{\Gamma} \to \mathcal{C}$ is the mapping which associates to every $\varphi \in \overline{\Gamma}$ the maximally consistent set $\overline{\{\varphi\}}$. It is easy to show that $\mathcal{C}$ is covered by $b$.
    \item $\nu : P \to \mathcal{C}$ is the mapping defined by: $p \mapsto \left\{
    \begin{array}{ll}
    \overline{\Gamma} & \mbox{if $p \in \overline{\Gamma}$} \\
    \emptyset & \mbox{otherwise}
    \end{array} \right.$
\end{itemize}
\end{definition}

\begin{proposition}
\label{prop:existence model}
For every $\varphi \in \mathcal{F}$, we have either $\sem{\mathcal{M}_\Gamma}(\varphi) = \overline{\Gamma}$ and $\varphi \in \overline{\Gamma}$, or $\sem{\mathcal{M}_\Gamma}(\varphi) = \emptyset$ and $\varphi \notin \overline{\Gamma}$.
\end{proposition}

\begin{proof}
By structural induction on the formula $\varphi$. The basic case as well as the cases where $\varphi$ is of the form $\varphi_1 \wedge \varphi_2$ and $\varphi_1 \vee \varphi_2$ are obvious. The other cases are treated as follows:

\begin{itemize}
    \item $\varphi$ is of the form $\varphi_1 \Rightarrow \varphi_2$. By definition, we have that
    $$\sem{\mathcal{M}_\Gamma}(\varphi) = \bigvee \{\overline{\Gamma}' \mid \overline{\Gamma}' \wedge \sem{\mathcal{M}_\Gamma}(\varphi_1) \subseteq \sem{\mathcal{M}_\Gamma}(\varphi_2)\}$$
    Here, by the induction hypothesis several cases have to be considered:
    \begin{enumerate}
        \item $\sem{\mathcal{M}_\Gamma}(\varphi_1) = \emptyset$  and $\sem{\mathcal{M}_\Gamma}(\varphi_2) = \emptyset$. In this case, we necessarily have that $\sem{\mathcal{M}_\Gamma}(\varphi) = \overline{\Gamma}$. 
        \item $\sem{\mathcal{M}_\Gamma}(\varphi_1) = \emptyset$  and $\sem{\mathcal{M}_\Gamma}(\varphi_2) = \overline{\Gamma}$. In this case, we necessarily have that $\sem{\mathcal{M}_\Gamma}(\varphi) = \overline{\Gamma}$. 
        \item $\sem{\mathcal{M}_\Gamma}(\varphi_1) = \overline{\Gamma}$  and $\sem{\mathcal{M}_\Gamma}(\varphi_2) = \emptyset$. In this case, we necessarily have that $\sem{\mathcal{M}_\Gamma}(\varphi) = \emptyset$. 
        \item $\sem{\mathcal{M}_\Gamma}(\varphi_1) = \overline{\Gamma}$  and $\sem{\mathcal{M}_\Gamma}(\varphi_2) = \overline{\Gamma}$. In this case, we necessarily have that $\sem{\mathcal{M}_\Gamma}(\varphi) = \overline{\Gamma}$. 
    \end{enumerate}
    In all the case, as $\overline{\Gamma}$ is maximally consistent, we easily deduce that $\varphi$ belongs or not to $\overline{\Gamma}$ depending on the case. 
    \item $\varphi$ is of the form $\neg \psi$. By definition, we have that
    $$\sem{\mathcal{M}_\Gamma}(\varphi) = \bigvee \{\overline{\Gamma}' \mid \overline{\Gamma}' \wedge \sem{\mathcal{M}_\Gamma}(\psi) \subseteq \emptyset\}$$
    Here, by the induction hypothesis, two cases have to be considered:
    \begin{enumerate}
        \item $\sem{\mathcal{M}_\Gamma}(\psi) = \emptyset$. We then have that $\sem{\mathcal{M}_\Gamma}(\varphi) = \overline{\Gamma}$.
        \item $\sem{\mathcal{M}_\Gamma}(\psi) = \overline{\Gamma}$. We then have that $\sem{\mathcal{M}_\Gamma}(\varphi) = \emptyset$.
    \end{enumerate}
    As $\overline{\Gamma}$ is maximally consistent, deducing whether $\varphi$ belongs or not to $\overline{\Gamma}$ depending on the case is also obvious.
    \item $\varphi$ is of the form $\Box \psi$. By definition, we have that
    $$\sem{\mathcal{M}_\Gamma}(\varphi) = \bigvee \{\overline{\Gamma}' \mid \forall \theta \in \overline{\Gamma}', \overline{\{\theta\}} \subseteq \sem{\mathcal{M}_\Gamma}(\psi)\}$$
    Here, by the induction hypothesis, two cases have to be considered:
    \begin{enumerate}
        \item $\sem{\mathcal{M}_\Gamma}(\psi) = \emptyset$. We then have that $\sem{\mathcal{M}_\Gamma}(\varphi) = \emptyset$.
        \item $\sem{\mathcal{M}_\Gamma}(\psi) = \overline{\Gamma}$. We then have that $\sem{\mathcal{M}_\Gamma}(\varphi) = \overline{\Gamma}$ (cf. Theorem~\ref{th:main results}).
    \end{enumerate}
    As $\overline{\Gamma}$ is maximally consistent, deducing whether $\varphi$ belongs or not to $\overline{\Gamma}$ depending on the case is direct by the Necessity rule.
    \item $\varphi$ is of the form $\Diamond \psi$. By definition, we have that
    $$\sem{\mathcal{M}_\Gamma}(\varphi) = \bigwedge \{\overline{\Gamma}' \mid \forall \theta \in \sem{\mathcal{M}_\Gamma}(\psi), \overline{\{\theta\}} \subseteq \overline{\Gamma}'\}$$
    Here, by the induction hypothesis, two cases have to be considered:
    \begin{enumerate}
        \item $\sem{\mathcal{M}_\Gamma}(\psi) = \emptyset$. We then have that $\sem{\mathcal{M}_\Gamma}(\varphi) = \emptyset$ (cf. Theorem~\ref{th:main results}).
        \item $\sem{\mathcal{M}_\Gamma}(\psi) = \overline{\Gamma}$. We then have that $\sem{\mathcal{M}_\Gamma}(\varphi) = \overline{\Gamma}$.
    \end{enumerate}
    As $\overline{\Gamma}$ is maximally consistent, deducing whether $\varphi$ belongs or not to $\overline{\Gamma}$ depending on the case is direct by the extensivity axiom and Modus Ponens.
\end{itemize}
\end{proof}

\begin{corollary}
\label{existence of a model}
$\Gamma \vdash \varphi \Longleftrightarrow \mathcal{M}_\Gamma \models \varphi$
\end{corollary}

\begin{proof}
Direct consequence of Proposition~\ref{prop:existence model}.
\end{proof}

\begin{theorem}[Completeness]
For every $\Gamma \subseteq \mathcal{F}$ and every $\varphi \in \mathcal{F}$, we have that:
$$\Gamma \models \varphi \Longrightarrow \Gamma \vdash \varphi$$
\end{theorem}

\begin{proof}
If $\Gamma {\not \vdash} \varphi$, then $\Gamma \cup \{\neg \varphi\}$ is consistent. By Proposition~\ref{lindenbaum lemma}, there exists a maximal consistent set of formulas $\overline{\Gamma}$ that extends $\Gamma$, and then by Corollary~\ref{existence of a model}, we have that  that $\mathcal{M}_\Gamma \models \neg \varphi$, i.e. $\mathcal{M}_\Gamma {\not \models} \varphi$. 
\end{proof}

If we restrict ourselves to the presheaf topos of sets (i.e. the topos of presheaves of the form $F : \bullet \to Set$), we then find the standard normal modal logic T and its extensions S4, B and S5 if we add the following axioms:

\begin{itemize}
    \item $\Box \varphi \Rightarrow \Box \Box \varphi$ (S4)
    \item $\varphi \Rightarrow \Box \Diamond \varphi$ (B)
    \item $\Diamond \varphi \Rightarrow \Box \Diamond \varphi$ (S5)
\end{itemize}

In this case, models are of the form $((\mathcal{P}(Q),\subseteq),b,\nu)$ where $Q$ is the set of states, $b : Q \to \mathcal{P}(Q)$ is the accessibility relation, and $\nu : P \to \mathcal{P}(Q)$ is the evaluation of propositional variables, and then morpho-categories under consideration are Boolean. Hence, the defined logic is not intuitionistic anymore but classical. The calculus presented above has then to be transformed by accepting as tautologies all propositional tautology instances, and by adding the following axiom:

\begin{itemize}
    \item {\em Duality:} $\neg \Box \varphi \Leftrightarrow \Diamond \neg \varphi$
\end{itemize}
The proof of completeness is then in all respects equivalent to that above.
\section{Conclusion}

Deterministic mathematical morphology dealing with increasing operations relies on two core notions: the one of adjunction that induces powerful algebraic properties of the operators, and the one of structuring element, that allows defining concrete instantiations of abstract operators, useful in many applications on various structures such as images or more generally spatial information, graphs and hypergraphs, etc. In this paper, we proposed an additional level of abstraction by moving to the domain of categories. We introduced the notion of morpho-categories, on which dilations and erosions based on structuring elements (defined abstractly) are built. We showed that such categories can be defined from topos of presheaves, which are general enough to include a number of structures (sets, hypergraphs, etc.). Then in order to be able to handle inclusions between subobjects, we introduced the notion of morpholizable category defined from a family of faithful functors associated with the category. Inclusion morphisms were then defined. Besides the definitions and theoretical results proved in the paper, examples of structuring elements in this new framework and associated dilations and erosions illustrated that the proposed construction applies to sets, graphs, hypergraphs, simplicial complexes, and includes previous definitions on such structures. Hence the proposed abstraction encompasses each particular case, and can be instantiated in many different ways. {Finally, we applied our abstract framework to define an intuitionistic modal logic by giving semantics of the modalities $\Box$ and $\Diamond$ from erosion and dilation, respectively.}

Future work will aim at further investigating the capabilities of the proposed framework, for example to deal with imprecision, based on fuzzy representations, or to handle extensions of simplicial complexes, such as simplicial sets (which can be defined by presheaves $F : \Delta^{op} \to Set$ where $\Delta$ is the simplex category, i.e. the category of finite ordinals with order-preserving functions as morphisms), in a algebraic topology context. {Another extension would be to study the properties of the modal logic defined in the paper (e.g. Henessy-Milner theorem, finite model property, decidability results)  and also to explore its links with spatial reasoning. For this, we could adapt and continue the approach that we initiated in~\cite{AB19}.}

%\bibliographystyle{plainnat}
%\bibliography{biblio}

\end{document}